\documentclass[reqno]{amsart}
\usepackage{amssymb, amsmath, enumerate}
\usepackage{graphicx} 
\usepackage[utf8x]{inputenc}

\newcommand\cal{\mathcal}
\def\qmod#1#2{{\hbox{}^{\displaystyle{#1}}}\!\big/\!\hbox{}_{
\displaystyle{#2}}}

\def\resto#1#2{{
#1\hskip 0.4ex\vline_{\hskip 0.2ex\raisebox{-0,2ex}
{{${\scriptstyle #2}$}}}}}

\def\C{{\mathbb C}}

\def\P{{\mathbb P}}

\def\R{{\mathbb R}}

\def\Z{{\mathbb Z}}

\def\union{\mathop{\bigcup}}

\def\textmap#1{\mathop{\vbox{\ialign{
                                 ##\crcr
     ${\scriptstyle\hfil\;\;#1\;\;\hfil}$\crcr
     \noalign{\kern 1pt\nointerlineskip}
     \rightarrowfill\crcr}}\;}}

\def\textlmap#1{\mathop{\vbox{\ialign{
                                 ##\crcr
     ${\scriptstyle\hfil\;\;#1\;\;\hfil}$\crcr
     \noalign{\kern-1pt\nointerlineskip}
     \leftarrowfill\crcr}}\;}}

\newfam\meuffam

\font\tenmeuf=eufm10
\font\sevenmeuf=eufm7
\font\fivemeuf=eufm5

\textfont\meuffam=\tenmeuf
\scriptfont\meuffam=\sevenmeuf
\scriptscriptfont\meuffam=\fivemeuf

\def\germ{\fam\meuffam\tenmeuf}

\def\g{{\germ g}}
\def\hg{{\germ h}}

\def\kg{{\germ k}}
\def\lg{{\germ l}}

\def\Ag{{\germ A}}

\def\Ig{{\germ I}}

\newtheorem{sz}{Satz}[section]
\newtheorem{thry}[sz]{Theorem}
\newtheorem{pr}[sz]{Proposition}
\newtheorem{re}[sz]{Remark}
\newtheorem{co}[sz]{Corollary}

\newtheorem{lm}[sz]{Lemma}

\begin{document}
\def\Pr{{\rm Pr}}
\def\tr{{\rm Tr}}
\def\End{{\rm End}}
\def\Aut{{\rm Aut}}
\def\Spin{{\rm Spin}}
\def\U{{\rm U}}
\def\SU{{\rm SU}}
\def\SO{{\rm SO}}
\def\PU{{\rm PU}}
\def\GL{{\rm GL}}
\def\spin{{\rm spin}}
\def\su{{\rm su}}
\def\so{{\rm so}}
\def\ub{\underbar}
\def\pu{{\rm pu}}
\def\Pic{{\rm Pic}}
\def\Iso{{\rm Iso}}
\def\NS{{\rm NS}}
\def\deg{{\rm deg}}
\def\Hom{{\rm Hom}}
\def\Aut{{\rm Aut}}
\def\h{{\germ h}}
\def\Herm{{\rm Herm}}
\def\Vol{{\rm Vol}}
\def\pf{{\bf Proof: }}
\def\id{{\rm id}}
\def\i{{\germ i}}
\def\im{{\rm im}}
\def\rk{{\rm rk}}
\def\ad{{\rm ad}}
\def\h{{\bf H}}
\def\coker{{\rm coker}}
\def\dbar{\bar{\partial}}
\def\Lo{{\Lambda_g}}
\def\niq{=\kern-.18cm /\kern.08cm}
\def\Ad{{\rm Ad}}
\def\RSU{\R SU}
\def\ad{{\rm ad}}
\def\dva{\bar\partial_A}
\def\da{\partial_A}
\def\p{\partial\bar\partial}
\def\sp{\Sigma^{+}}
\def\sm{\Sigma^{-}}
\def\spm{\Sigma^{\pm}}
\def\smp{\Sigma^{\mp}}
\def\Tors{{\rm Tors}}
\def\st{{\rm st}}
\def\s{{\rm s}}
\def\oo{{\scriptstyle{\cal O}}}
\def\ooo{{\scriptscriptstyle{\cal O}}}
\def\sw{Seiberg-Witten }
\def\pa{\partial_A\bar\partial_A}
\def\Dr{{\raisebox{0.15ex}{$\not$}}{\hskip -1pt {D}}}
\def\gr{{\scriptscriptstyle|}\hskip -4pt{\g}}
\def\subsetint{{\  {\subset}\hskip -2.45mm{\raisebox{.28ex}
{$\scriptscriptstyle\subset$}}\ }}
\def\ra{\rightarrow}
\def\pst{{\rm pst}}
\def\sst{{\rm sst}}

\def\kod{{\rm kod}}
\def\degmax{{\rm degmax}}
\def\red{{\rm red}}
\def\ev{\mathrm{ev}}
\def\TE{{^{\cal T}}\hspace{-1.1mm}{\cal E}}
\def\TS{ {^{\cal T}}\hspace{-1.1mm}{\cal S}}

\title[Variation formula and compact subspaces of ${\cal M}^\st$]{A variation formula for the determinant line bundle. Compact subspaces of moduli spaces of stable bundles over class VII surfaces}

\author{Andrei Teleman}
\date\today 
\thanks{The author has been partially supported by the ANR project MNGNK, decision 
Nr.  ANR-10-BLAN-0118}
\address{CMI, LATP,  Aix-Marseille Université,   39  Rue F. Joliot-Curie,  13453
Marseille Cedex 13, France} \email{teleman@cmi.univ-mrs.fr}

\begin{abstract}
This article deals with two topics: the first, which has a  general character, is a variation formula for the the determinant line bundle  in non-Kählerian geometry. This formula, which is a consequence of the non-Kählerian version of the Grothendieck-Riemann Roch theorem proved recently by Bismut \cite{Bi}, gives the variation of the determinant line bundle corresponding to a perturbation of a Fourier-Mukai kernel ${\cal E}$ on a product $B\times X$  by a unitary flat line bundle on the fiber $X$. When this fiber is a complex surface and ${\cal E}$ is a   holomorphic 2-bundle, the result can be interpreted as a Donaldson invariant.  

The second topic concerns a geometric application of our variation formula, namely we will study compact complex subspaces  of  the moduli spaces of stable bundles   considered in  our program for proving existence of curves on minimal class VII surfaces \cite{Te4}.  Such a moduli space comes with a distinguished point $a=[{\cal A}]$ corresponding to the canonical extension ${\cal A}$ of $X$ \cite{Te3}, \cite{Te4}. The compact subspaces $Y\subset {\cal M}^\st$ containing this distinguished point play an important role in our program. We will prove a non-existence result:  there exists no compact complex subspace of positive dimension $Y\subset {\cal M}^\st$ containing $a$   with an open neighborhood $a\in Y_a\subset Y$   such that $Y_a\setminus\{a\}$ consists only of non-filtrable bundles. In other words, within any compact complex subspace of positive dimension $Y\subset {\cal M}^\st$ containing $a$, the point $a$ can be approached by filtrable bundles.   Specializing to the case $b_2=2$ we obtain a new way to complete the proof of the main result of \cite{Te4}: any minimal class VII surface with $b_2=2$ has a cycle of curves. Applications to class VII surfaces with higher $b_2$ will be  be discussed in a forthcoming article.\\

\end{abstract}

\maketitle

\tableofcontents

\section{Introduction}\label{intro}

 Let $B$, $X$ be compact complex manifolds, $p:B\times X\to B$, $q:B\times X\to X$ the two projections and   ${\cal E}\in \mathrm{Coh}(B\times X)$. The Fourier-Mukai transform of kernel ${\cal E}$ is the functor  $\phi_{\cal E}:\mathrm{Coh}(X)\to \mathrm{Gr}(\mathrm{Coh}(B))
$ from the category of coherent sheaves on $X$ to the category  of graded coherent sheaves on $B$ defined by
$$\phi_{\mathcal E}({\cal F}):=R^\bullet p_*(\mathcal{E}\otimes  q^*({\mathcal F}))\ . 
 \eqno{(FM)}
$$ 
In projective algebraic geometry  $\phi_{\mathcal E}$ can be lifted to a functor 
$$\Phi_{\mathcal {E}}:\mathrm{D}^b(X)\to \mathrm{D}^b(B)
$$ 
between   the bounded derived categories of  coherent sheaves on the two manifolds, and has sense  for a kernel ${\cal E}\in \mathrm{D}^b(B\times X)$.
Such functors   are extensively used in the literature in order to compare the derived categories associated with two projective varieties. 

The non-Kählerian version of the Grothendieck-Riemann Roch theorem \cite{Bi}    computes the Chern character   $\mathrm{ch}(\phi_{\mathcal E}({\cal F})) $ in terms of the Chern classes of the kernel $\cal{E}$ and of the variable sheaf $\cal F$ {\it in the Bott-Chern cohomology} 
of $B$  assuming that  $\cal{E}$, $\cal F$ and the direct images $R^i p_*(\mathcal{E}\otimes  q^*({\mathcal F}))$ are locally free.   Conjecturally these assumptions are not necessary.

We denote by ${\cal L}(X)$, ${\cal L}(B)$ the categories of holomorphic line bundles (with isomorphisms as morphisms) on $X$ and $B$ respectively. We will use the simpler functor $\delta_{\cal E}:{\cal L}(X)\to {\cal L}(B)$ given by
$$\delta_{\cal E}({\cal F}):=\det(R^\bullet p_*(\mathcal{E}\otimes q^*({\mathcal F})))=\lambda_{\mathcal{E}\otimes q^*(\mathcal {F})}
$$  
obtained by composing $\phi_{\cal E}$ with the determinant functor $\lambda$ \cite{Bi}. The Chern class of  $\delta_{\cal E}({\cal F})$ in Bott-Chern cohomology can be computed  using Bismut results (see Theorem 11.8.1 \cite{Bi}) in full generality, without any assumption on the two manifolds $X$, $B$ or the direct images  $R^i p_*(\mathcal{E}\otimes  q^*({\mathcal L}))$.

The functor $\delta_{\cal E}$ can be regarded as a method for     constructing   holomorphic line bundles on an   unknown manifold  $B$ using a fixed kernel ${\cal E}\in\mathrm{Coh}( B\times X)$ and  variable line bundles  ${\cal L}$ on the known manifold $X$.
Passing to isomorphism classes  one obtains a holomorphic map  
$$[\delta_{\cal E}]:\Pic(X)\to {\rm Pic}(B)\ ,\ [{\cal L}]\mapsto  
[\lambda_{\mathcal{E}\otimes q^*({\mathcal L})}]\in\Pic(B)\ .$$
between Abelian complex groups. The holomorphy of $[\delta_{\cal E}]$ is a consequence of the base change theorem for the determinant line bundle in complex geometry (see Proposition 2.4 \cite{Gri}).

The  starting point of our first  topic is the natural questions: What can be said about this holomorphic map? The question is particularly interesting in the non-Kählerian framework  because  the connected component $\Pic^0(X)$ of the origin in $\Pic(X)$ is not necessary compact, so a priori $[\delta_{\cal E}](\Pic^0(X))$ might be a very complicated non-closed subset of $\Pic(B)$. Although very natural, the problem is difficult, because even the simple operation of perturbing the kernel ${\cal E}$ with an $n$-th root of $[{\cal O}_X]\in\Pic^0(X)$ can change dramatically the direct images $R^i p_*(\mathcal{E})$.

We will prove a variation formula, which computes the differential  of the restriction of $[\delta_{\cal E}]$ to the subgroup of unitary flat line bundles. In other words our variation formula computes the infinitesimal variation of $[{\cal L}_{\cal E}]\in\Pic(B)$ when the kernel  ${\cal E}$ is perturbed by a unitary flat line bundle on the fiber $X$. We will see that in the particular case when $X$ is a complex surface and ${\cal E}$ is a holomorphic 2 bundle satisfying a technical condition,  the result can be interpreted as a Donaldson invariant.
We acknowledge that a similar variation formula  has recently been proved by Grivaux \cite{Gri}.    Grivaux's formula is an identity in Deligne cohomology, which gives  the differential of $[\delta_{\cal E}]$ on {\it whole} $\Pic(X)$ as  the fibre integration of a Deligne cohomology class on $B\times X$. Fibre integration in Deligne cohomology is a difficult operation.  Our variation formula is slightly weaker, but it makes use only of classical fibre integration of forms, which allowed us to express the result as a Donaldson invariant.
\\

The second part of the article gives an application of our variation formula, application which concerns the geometry of the moduli spaces intervening  in our program for proving existence of curves on class VII surfaces. Let $X$ be a minimal class VII surface with $b_2(X)>0$, $E$ a differentiable rank 2-bundle on $X$ with $c_2(E)=0$ and $\det(E)=K_X$ (the underlying differentiable line bundle of the canonical line bundle ${\cal K}_X$), and $g$ a Gauduchon metric on $X$. We denote by ${\cal M}^\st$ the moduli space of stable structures on $X$ inducing the holomorphic structure ${\cal K}_X$ on $\det(E)$ modulo the complex gauge group $\Gamma(X,\mathrm{SL}(E))$.  By definition the canonical extension  of $X$ is the essentially unique non-trivial extension of the form
\begin{equation}\label{A}
0\to {\cal K}_X\textmap{i_0} {\cal A}\textmap{p_0} {\cal O}_X\to 0\ .
\end{equation}

Excepting for very special Kato surfaces, the canonical extension ${\cal A}$ is stable for suitable Gauduchon metrics \cite{Te3}, \cite{Te4}. Therefore, choosing such a Gauduchon metric we obtain a distinguished point $a=[{\cal A}]\in {\cal M}^\st$.  We shall prove that there exists no compact complex subspace of positive dimension $Y\subset {\cal M}^\st $ containing $a$   with an open neighborhood $a\in Y_a\subset Y$   such that $Y_a\setminus\{a\}$ consists only of non-filtrable bundles. In other words, within any compact complex subspace of positive dimension $Y\subset {\cal M}^\st$ containing $a$, the point $a$ can be approached by filtrable bundles. This result can be regarded as a version --  valid  for minimal class VII surfaces with arbitrary $b_2>0$ -- of   Corollary 5.3 of \cite{Te2} proved for minimal class VII surfaces with $b_2=1$.  The   proof uses our variation formula, a Leray spectral sequence argument and the results in \cite{Te6} on the torsion of the first direct image of a locally free sheaf.

At the end of the paper we explain how this non-existence theorem can be used to give a new proof of a difficult result used in \cite{Te4} for proving the existence of a cycle on class VII surfaces with $b_2=2$.  Applications to class VII surfaces with higher $b_2$ will be  be discussed in a forthcoming article.

\section{A variation formula for the determinant line bundle in non-Kählerian geometry}
\label{variation}

\subsection{Unitary flat line bundles and the degree map in non-Kähler geometry}
Let $X$ be a compact complex manifold. Via   the de Rham and Dolbeault isomorphisms the canonical $\R$-linear map  $\hg:H^1(X,i\R)\to H^1({\cal O}_X)$ is given by
\begin{equation}\label{hg}\hg([\alpha])=[\alpha^{0,1}]\ \forall \alpha\in  Z^1(X,i\R)\ .
\end{equation}
This map is injective (see Lemma 2.3 \cite{Te3}). This implies that the canonical map $2\pi iH^1(X, \Z)\to H^1({\cal O}_X)$ (which obviously factorizes through $\hg$) is also injective and has closed image, hence the Abelian group $\Pic^0(X)\simeq H^1({\cal O}_X)/2\pi i H^1(X,\Z)$ of isomorphism classes of topologically trivial holomorphic line bundles on $X$ has a natural complex Lie group structure induced from $H^1({\cal O}_X)$. 
In the general non-Kählerian framework this Abelian complex Lie group is not necessarily compact.

  Recall that we have a canonical short  exact sequence
\begin{equation}\label{NSES} 0\to\Pic^0(X)\hookrightarrow\Pic(X)\textmap{c_1} \NS(X)\to 0
\end{equation}
where the Neron-Severi group $\NS(X)$ of $X$ is defined by
$$\NS(X):=\ker( H^2(X,\Z)\to H^{0,2}(X,\C))\ .$$
A 2-cohomology class belongs to $\NS(X)$ if and only if  its image in $H^2_{\rm DR}(X,\C)$  has a representative of type (1,1). The exact sequence (\ref{NSES}) shows that $\Pic^0(X)$  is the 
connected component   of $[{\cal O}_X]$ in  $\Pic(X)$. This connected complex Lie group fits in the following diagram with exact rows and exact columns
\begin{equation}\label{diagram}
\begin{array}{c}
\unitlength=1mm
\begin{picture}(90,45)(-44,-25)

\put(-22,15){\vector(0, -3){3}}
\put(1,15){\vector(0, -3){3}}
\put(-23,17){$0$}
\put(0,17){$0$}

\put(-40,8){0}
\put(-37,9){\vector(2,0){5}}
\put(-30,8){$\Pic^0_{\rm uf}(X)$}
\put(-13,9){\vector(2,0){5}}
\put(-7,8){$\Pic^0(X)$}
\put(8,11){$c_1^{\mathrm{\scriptscriptstyle BC}}$}
\put(9,9){\vector(2,0){5}}
\put(15,8){$H^{1,1}_{\rm BC}(X,\R)^0$}
\put(36,9){\vector(2,0){5}}
\put(43,8){0}
\put(-22,4){\vector(0, -3){3}}
\put(1,4){\vector(0, -3){3}}
\put(22,4){\line(0, -3){3}} \put(23,4){\line(0, -3){3}}
%
\put(-40,-4){0}
\put(-37,-3){\vector(2,0){5}}
\put(-30,-4){$\Pic_{\rm uf}(X)$}
\put(-13,-3){\vector(2,0){5}}
\put(-7,-4){$\Pic^{\rm T}(X)$}
\put(8,-1){$c_1^{\mathrm{\scriptscriptstyle BC}}$}
\put(9,-3){\vector(2,0){5}}
\put(15,-4){$H^{1,1}_{\rm BC}(X,\R)^0$}
\put(36,-3){\vector(2,0){5}}
\put(43,-4){0}
\put(-22,-7){\vector(0, -3){4}}
\put(1,-7){\vector(0, -3){4}}

\put(-21,-9){$c_1$}
\put(2,-9){$c_1$}

%
\put(-36,-16){${\rm Tors}H^2(X,\Z)$}
\put(-12,-16){$=$}
\put(-7,-16){${\rm Tors}H^2(X,\Z)$}

\put(-22,-18){\vector(0, -3){3}}
\put(1,-18){\vector(0, -3){3}}
 \put(-23,-25){$0$}

\put(0,-25){$0$}

\end{picture} 
\end{array} \ ,
\end{equation}
where  $H^{1,1}_{\rm BC}(X,\R)$ is the space of real Bott-Chern classes of type (1,1) (see \cite{BHPV}, p. 148), $H^{1,1}_{\rm BC}(X,\R)^0$ is defined by
$$H^{1,1}_{\rm BC}(X,\R)^0:=\ker\big(H^{1,1}_{\rm BC}(X,\R)\to H^{1,1}_{\rm DR}(X,\R)\big)\ ,$$
and  $\Pic_{\rm uf}(X)$ is the subgroup of holomorphic line bundles on $X$ which admit a compatible flat unitary connection. Equivalently, fixing a Hermitian metric $g$ on $X$, a point $l=[{\cal L}]\in\Pic(X)$ belongs to $\Pic_{\rm uf}(X)$ if and only if the Chern connection of the (essentially unique) Hermite-Einstein metric $h$ on ${\cal L}$ (see Corollary 2.1.6 \cite{LT}) is flat.
  The {\it maximal compact subgroup} of $\Pic^0(X)$ is the connected component $\Pic^0_{\rm uf}(X)$ of the unit in $\Pic_{\rm uf}(X)$. One has natural identifications
 $$\Pic_{\rm uf}(X)= \Hom(\pi_1(X,x_0),S^1),\ 
\Pic^0_{\rm uf}(X)=\{\rho\in\Hom(\pi_1(X,x_0),S^1)|\  c_1({\cal L}_\rho)=0\},
$$
where, in general, for $\rho\in\Hom(\pi_1(X,x_0),\C^*)$,  ${\cal L}_\rho$ denotes the flat  holomorphic line bundle associated with $\rho$.  

The exactness of the first horizontal sequence in (\ref{diagram}) follows from  
\begin{lm}\label{Bott-Chern-Curv}
Let $X$ be a compact complex manifold, and ${\cal L}$ a holomorphic line bundle. Then
\begin{enumerate}[(i)]
\item One has 
$$c_1^{\mathrm{BC}}({\cal L})=\left\{\frac{i}{2\pi} F_{{\cal L},h}\vert\  h\hbox{ Hermitian metric on }{\cal L}\right\}$$ 
(as subsets of $A^{1,1}_\R(X)$), where $F_{{\cal L},h}\in iA^{1,1}_\R(X)$ denotes the curvature form of the Chern connection associated with the pair $({\cal L},h)$. 
\item One has $[{\cal L}]\in \Pic_{\rm uf}(X)$ if and only if $c_1^{\mathrm{BC}}({\cal L})=0$.
\end{enumerate}
 \end{lm}
 \begin{proof}
(i)  By definition $c_1^{\mathrm{BC}}({\cal L})$ is the Bott-Chern class of the Chern form %
 $$c_1({\cal L},h):=\frac{i}{2\pi} F_{{\cal L},h}$$
  of  any Hermitian metric $h$ on ${\cal L}$, hence the right hand set is contained in the left hand set. On the other  hand for any  smooth function $\varphi\in{\cal C}^\infty(X,\R)$ one has
 $$c_1({\cal L},e^\varphi h)=c_1({\cal L},h)+\frac{i}{2\pi} \bar\partial\partial\varphi\ ,
 $$
 which shows that, conversely, any representative of the Bott-Chern class $c_1^{\mathrm{BC}}({\cal L})$ is the Chern form of a suitable Hermitian metric on ${\cal L}$.
 \\ \\
 (ii) Follows from (i) and the definition of $\Pic_{\rm uf}(X)$.
 \end{proof}

\begin{re}\label{Lie}  The Lie algebra of $\Pic^0_{\rm uf}(X)$ is $\mathrm{Lie}(\Pic^0_{\rm uf}(X))=iH^1(X,\R)$ identified with a subspace of $H^1({\cal O}_X)=\mathrm{Lie}(\Pic^0(X))$ via the map $\hg$ defined above.  \\ 
\end{re}
We recall that a Hermitian metric $g$ on a complex $n$-manifold $X$ is called {\it Gauduchon} if $dd^c(\omega_g^{n-1})=0$. The  degree  map
$$\deg_g:\Pic(X)\to\R$$
associated with such a metric  is defined by  
$$\deg_g({\cal L}):=\int_X c_1({\cal L},h)\wedge \omega_g^{n-1}\ , $$
where  $h$  is a Hermitian metric on ${\cal L}$. Since the Bott-Chern class of the form $c_1({\cal L},h)$ is independent of $h$ and $dd^c(\omega_g^{n-1})=0$, the integral on the right is independent of $h$, hence $\deg_g$ is well defined.  We recall (see Proposition 1.3.13 \cite{LT}) that

\begin{re} If $X$ is a complex surface with $b_1(X)$ odd,  then $\deg_g$ is surjective on $\Pic^0(X)$,  and $\Pic^0_{\rm uf}(X)=\ker(\resto{\deg_g}{\Pic^0(X)}:\Pic^0(X)\to\R)$. If $X$ is a complex surface with $b_1(X)$ even, then $\Pic^0_{\rm uf}(X)=\Pic^0(X)$ is a compact torus and $\deg_g$ vanishes on this torus (even if $g$ is non-Kähler).
\end{re}
Therefore, 
\begin{re} For complex surface $X$ with $b_1(X)$ odd, one has $H^{1,1}_{\rm BC}(X,\R)^0\simeq\R$, and the maximal compact subgroup $\Pic^0_{\rm uf}(X)$ of $\Pic^0(X)$ is just the real hypersurface cut out by the Gauduchon degree.
\end{re}
\vspace{-1mm}{\ }\\
{\bf Example:} (see \cite{Te2}) Let $X$ be a class VII surface, i.e. a complex surface $X$ with $b_1(X)=1$ and $\kod(X)=-\infty$. Then 
\begin{enumerate}
\item $\NS(X)=H^2(X,\Z)$,
\item The composition of the  isomorphism 
$$\C^*\to \Hom\left(\qmod{H_1(X,\Z)}{\Tors},\C^*\right)$$
defined by a generator $u$ of $H_1(X,\Z)/{\Tors}\simeq\Z$ with the obvious morphism
$$\ \ \ \ \ \ \ \ \ \ \  \    \Hom\left(\qmod{H_1(X,\Z)}{\Tors},\C^*\right)\to \Hom(H_1(X,\Z),\C^*)=H^1(X,\C^*)\to H^1({\cal O}^*_X)
$$ 
defines an isomorphism $l_u:\C^*\textmap{\simeq} \Pic^0(X)$. Choosing $u$ in a convenient way we have for any Gauduchon metric $g$ on $X$
$$\deg_g(l(\zeta))=C_g \log|\zeta| \ \forall \zeta\in\C^*\ ,
$$
where $C_g$ is a positive constant depending smoothly on $g$.
\end{enumerate}
\vspace{5mm}

\subsection{The variation formula in the unitary flat directions} 

We begin with the following simple

\begin{re} \label{uf-pert} Let $X$, $B$ be compact complex manifolds  and ${\cal E}$ a holomorphic vector bundle on $B\times X$. Then one has
$$  \lambda_{{\cal E} \otimes q^*( \Pic_{\rm uf}(X))}\subset   \lambda_{\cal E} \otimes \Pic_{\rm uf}(B)\ ,\ \lambda_{{\cal E} \otimes q^*( \Pic^0_{\rm uf}(X))}\subset   \lambda_{\cal E} \otimes \Pic^0_{\rm uf}(B)\ ,$$
  hence  perturbing the kernel ${\cal E}$   by a (topologically trivial) unitary flat line bundle on the fiber $X$ will change the determinant line bundle by a (topologically trivial)  unitary flat line bundle on  $B$. 
\end{re} 
 
\begin{proof} 
Fix a Hermitian metric $h$ on ${\cal E}$ and let ${\cal L}$ be a holomorphic line bundle on $X$ whose isomorphy class belongs to $\Pic_{\rm uf}(X)$. In other words ${\cal L}$ admits a Hermitian metric $h_0$ such that the associated Chern connection is flat. The Chern forms of ${\cal E}\otimes q^*({\cal L})$ computed using the tensor product metric  $h\otimes q^*(h_0)$ coincide with the Chern forms of ${\cal E}$. Using Theorem 11.8.1 of \cite{Bi} we see that  
$$c_1^{\mathrm{BC}}(\lambda_{{\cal E} \otimes q^*({\cal L})})=c_1^{\mathrm{BC}}(\lambda_{{\cal E} })\ ,
$$ 
therefore, by  Lemma \ref{Bott-Chern-Curv}, the holomorphic line bundles ${\cal D}:=\lambda_{\cal E}$,  ${\cal D}':=\lambda_{{\cal E} \otimes q^*({\cal L})}$ on $B$ admit Hermitian metrics $\chi$, $\chi'$ such that 
$$F_{{\cal D},\chi}= F_{{\cal D}',\chi'}\ .
$$
This shows that ${\cal D}'\otimes {\cal D}^\vee$ admits a Hermitian metric whose  Chern connection is flat, hence $[{\cal D}'\otimes {\cal D}^\vee]\in \Pic_{\rm uf}(B)$ as claimed.
The similar statement for $\Pic^0_{\rm uf}(X)$ is proved taking into account that the map $[\delta_{\cal E}]:\Pic(X)\to\Pic(B)$ is holomorphic (as we mentioned in section \ref{intro}), so continuous.  
\end{proof}
Taking into account Remark \ref{uf-pert} is natural to ask:
\begin{itemize}
\item Compute the linearization  $l_{\cal E}:iH^1(X,\R)\to iH^1(B,\R)$ of  the map  
$$\Pic^0_{\rm uf}(X) \ni [{\cal L}]\ \stackrel{[\delta_{\cal E}]}{\longmapsto} \ \left[\lambda_{{\cal E} \otimes q^*({\cal L})}\right]\in  \lambda_{\cal E} \otimes \Pic^0_{\rm uf}(B)\ .$$
\item  How does  $l_{\cal E}$  depend on the kernel ${\cal E}$? Is it a topological invariant? 
\item  Compare  $l_{\cal E}$ with $l_{{\cal E}\otimes q^*({\cal T})}$ for a line  bundle ${\cal T}\in\Pic(X)$.  
\end{itemize}
The linear map $l_{\cal E}:iH^1(X,\R)\to iH^1(B,\R)$ should be called   the linearization of $[\delta_{\cal E}]$  in the unitary flat directions.

\begin{thry}   \label{T1}  Let $X$, $B$ be compact complex manifolds  and ${\cal E}$ a holomorphic vector bundle on $B\times X$. For any $u\in iH^1(X,\R)$ one has
\begin{equation}\label{lE}
l_{\cal E}(u)=p_*\big(q^*(u)\cup \mathrm{ch}({\cal E})\cup\mathrm{td}(X)\big)^{(1)}\ ,
\end{equation}
where the exponent $^{(1)}$ on the right means  ``the term of degree 1".
\end{thry}
Therefore the linearization $l_{\cal E}$ is determined by the Chern classes of ${\cal E}$, so has a topological character.

\begin{proof}  

Let ${\cal P}$ be a Poincaré line bundle on $\Pic^0(X)\times X$. This means that   ${\cal P}$   has the following property: for every $l\in\Pic^0(l)$ the holomorphic line bundle ${\cal P}_l$ on $X$ defined by  restriction  of ${\cal P}$ to $\{l\}\times X$ belongs to the isomorphy class $l$.  Fix a Hermitian metric $g$ on $X$. Solving the Hermite-Einstein equation (see Corollary  2.1.6 \cite{LT}) fiberwise,  we obtain a Hermitian metric $H$ on ${\cal P}$ -- unique up to multiplication with a function $\rho\in {\cal C}^\infty(\Pic^0(X),\R_{>0})$ --  such that for every point $l\in\Pic^0(X)$ the Chern connection $A_l:=A_{{\cal P}_l,H_l}$ of the restriction $H_l$ of $H$ to   ${\cal P}_l$   is Hermite-Einstein.  For $l\in\Pic^0_{\rm uf}(X)$ the Chern connection $A_{l}$ is flat.  Let $u\in i H^1(X,\R)$ and let $v\in T_0^\R(\Pic^0_{\rm uf}(X))$ be the corresponding real tangent vector. Using Remark \ref{closed2} in the Appendix we see that the imaginary 1-form $\iota_v F_{{\cal P},H}$ on $X$ is closed and is a representative of the Rham class $u$. In other words
\begin{equation}\label{iotav}
[\iota_v F_{{\cal P},H}]_{\rm DR}=u\ .
\end{equation}

Consider the projection 
$$\tilde p: \Pic^0(X)\times B\times X\to \Pic^0(X)\times B=:\tilde B\ ,$$
let $\tilde{\cal D}$ be the determinant line bundle  (with respect to $\tilde p$) of  
$$\tilde{\cal E}:=p_{B\times X}^*({\cal E} )\otimes p_{\Pic^0(X)\times X}^*({\cal P})\ .$$  
and let ${\cal D}_l$ be the line bundle on $B$ defined by the restriction of ${\cal D}$ to $\{l\}\times B$. Since the determinant line bundle commutes with base change (Proposition 2.4 \cite{Gri}) we obtain  an isomorphism 
\begin{equation}\label{taut}
\lambda_{{\cal E}\otimes q^*({\cal L})}\simeq {\cal D}_{[{\cal L}]}\ ,
\end{equation}
for every topologically trivial holomorphic line bundle ${\cal L}$ on $X$. Therefore the variation $l_{\cal E}(u)$ is the derivative of the map $l\mapsto [{\cal D}_l]$ in the direction $v$ defined by $u$.  We will compute this derivative using Proposition \ref{diff-curv}  proved in the Appendix.
 
Fix a Hermitian metric  $h$ on ${\cal E}$.    Combining Bismut's  formula for the Chern class of the determinant line bundle   in Bott-Chern cohomology (Theorem 11.8.1   \cite{Bi}) with Lemma \ref{Bott-Chern-Curv} we obtain a Hermitian metric $\tilde\chi$ on $\tilde{\cal D}$ such that the curvature of the corresponding Chern connection  $\tilde A$ is
$$F_{\tilde A}=-2\pi i \tilde p_*\big(\mathrm{ch}({\cal P},H)\cup\mathrm{ch}({\cal E},h)\cup \mathrm{td}(X,g)\big)^{(2)}\in iA^{1,1}(\Pic^0(X)\times B)\ .
$$

Put $n:=\dim_\C(X)$. For the computation of the 1-form $\iota_v F_{\tilde A}$ on $B$ we need only the restriction to $\Pic^0_{\rm uf}(X)\times B\times X$ of the terms of degree $(1,1,2n)$ in the exterior polynomial $\mathrm{ch}({\cal P},H)\cup\mathrm{ch}({\cal E},h)\cup \mathrm{td}(X,g)$ on the product $\Pic^0(X)\times B\times X$. Since the restriction  of $F_{{\cal P},H}$ to $\Pic^0_{\rm uf}(X)\times X$ vanishes in the $X$-directions, the only term  of $\mathrm{ch}({\cal P},H)=e^{c_1({\cal P},H)}$ which contributes to $\iota_v F_{\tilde A}$ is the mixed term  of   
$$\resto{c_1({\cal P},H)}{\Pic^0_{\rm uf}(X)\times X}=\frac{i}{2\pi}\resto{F_{{\cal P},H}}{\Pic^0_{\rm uf}(X)\times X}\ .$$
Putting $\alpha:=\iota_v F_{{\cal P},H}\in iZ^1(X,\R)$ we obtain
$$\iota_v F_{\tilde A}= p_*\big(q^*(\alpha)\cup\mathrm{ch}({\cal E},h)\cup \mathrm{td}(X,g)\big)^{(1)}\in iZ^{1}(B)\ .
$$

 The claim follows now by Proposition \ref{diff-curv} taking into account (\ref{iotav}). 
\end{proof} 

\begin{re} \label{ConstantSpeed} In the conditions  of Theorem  \ref{T1} the following holds  
\begin{enumerate}[(i)]
\item  The map $l_{\cal E}$ has a topological character, in particular 
\begin{equation}\label{forPic0}
 l_{\cal E}=l_{{\cal E}\otimes q^*{\cal L}}\ \forall [{\cal L}]\in\Pic^0(X)\ .
 \end{equation}
\item If  $l_{\cal E}=0$ then $[\delta_{{\cal E}}]$ is constant on the compact real torus $\Pic^0_{\rm uf}(X)\subset\Pic^0(X)$. 
\item Suppose that $X$ is a  complex surface. In this case  $\Pic^0_{\rm uf}(X)$ either coincides with $\Pic^0(X)$ (when $b_1(X)$ is even) or is a real hypersurface of $\Pic^0(X)$ (when $b_1(X)$ is odd). Therefore, since $[\delta_{\cal E}]$ is holomorphic, we obtain in this case
\begin{equation}\label{le=0}
l_{\cal E}=0\Longleftrightarrow \resto{[\delta_{\cal E}]}{\Pic^0(X)}\hbox{ is constant}\ .
\end{equation}

\end{enumerate}
\end{re}

 There exists  another  natural approach  to prove the variation formula for the determinant line bundle: suppose we have  a general Grothendieck-Riemann-Roch formula for proper  holomorphic maps in {\it Deligne cohomology}. For a compact complex manifold $X$ the first Chern class map in Deligne cohomology  is just the map
$$[\det]:\mathrm{Coh} (X) \to\Pic (X)$$
 mapping a coherent sheaf to the isomorphism class of its determinant line bundle, hence such a Grothendieck-Riemann-Roch formula will compute $\lambda_{\cal E}$ up to isomorphism, not just an invariant of this holomorphic line bundle.    It suffices to differentiate the obtained formula for $\lambda_{{\cal E}\otimes q^*({\cal L})}$ with respect to ${\cal L}$. Unfortunately is seems that proving a Grothendieck-Riemann-Roch formula in Deligne cohomology is very hard, and for the moment we don't have the necessary tools to deal with this problem. Recent progress has been recently  obtained  by Julien Grivaux, who has recently obtained a variation formula in Deligne cohomology   although a  Grothendieck-Riemann-Roch formula in this  cohomology is not available yet.  His result can be regarded as a strong indication that such a general Grothendieck-Riemann-Roch formula in Deligne cohomology should hold.
 
For a cohomology class $u\in i H^1(X,\R)$ denote by $D_X(u)\in i H_{2n-1}(X,\R)$ its Poincaré dual homology class.

\begin{co} \label{Co1} Let $X$ be a complex surface and  ${\cal E}$ is a holomorphic rank $r$ bundle on $B\times X$ with $c_1({\cal E})\in p^*(H^2(B,\Z))+q^*(H^2(X,\Z))$
(so the  mixed  term in the Künneth decomposition  of  $c_1({\cal E})$ vanishes).   Then one has
$$l_{\cal E}(u)=-\frac{1}{2r} p_*(q^*(u)\cup c_2(\End_0({\cal E}))=-\frac{1}{2r}  c_2(\End_0({\cal E}))/D_X(u)
$$
In particular $l_{\cal E}=l_{{\cal E}\otimes q^*({\cal T})}$ for any holomorphic line bundle ${\cal T}$ on $X$.
\end{co}

\begin{proof} Since $\dim_\R(X)=4$  formula (\ref{lE}) shows that the only terms of $\mathrm{ch}({\cal E})$ which contribute to $l_{\cal E}(u)$ are those of bidegree $(1,1)$ and $(1,3)$ with respect to the Künneth decomposition of the cohomology of $B\times X$.  The hypothesis implies that in  the sum
$$\mathrm{ch}({\cal E})=r+c_1({\cal E})+\frac{1}{2}(c_1^2({\cal E})-2c_2({\cal E}))+\dots 
$$
neither $c_1({\cal E})$, nor $c_1^2({\cal E})$ contains terms of bidegree $(1,1)$ or $(1,3)$. Therefore the only term of $\mathrm{ch}({\cal E})$ which contributes to $l_{\cal E}(u)$ is the term of bidegree (1,3) of $c_2({\cal E})$. This shows that
 \begin{equation}\label{lEnew}
l_{\cal E}(u)=-p_*\big(q^*(u)\cup c_2({\cal E})\big)\ .
\end{equation}
It suffices to recall that
$$
c_2({\cal E}nd_0({\cal E}))=2rc_2({\cal E})-(r-1)c_1^2({\cal E})
$$
and to take into account again that, under our assumptions, $c_1^2({\cal E})$ does not contain terms of bidegree $(1,1)$ or $(1,3)$. For the second equality use the identity 5 p. 288 \cite{Sp} taking for $z$ the fundamental class of $X$.  
\end{proof}

\begin{re}
1. The same formula gives the Donaldson $\mu$-class associated with the homology class $D_X(u)\in H_3(X,\R)$ on a moduli space of irreducible $\mathrm{PU}(r)$ connections.  This gives an interesting geometric interpretation of this Donaldson class in Donaldson theory on complex surfaces.

2. The condition on the class $c_1({\cal E})$ assumed in the hypothesis of Corollary \ref{Co1} is satisfied in the gauge theoretical framework. Indeed, in Donaldson theory one uses moduli spaces of anti-selfdual $\mathrm{PU}(r)$-connections or, equivalently, moduli spaces of projectively anti-selfdual unitary connections with fixed determinant \cite{D}, \cite{Te2}, \cite{Te4}. The universal bundle associated with (a subspace of) such a moduli space, if exists,  satisfies this condition (see \cite{Te2}, \cite{Te4}  and Section \ref{CompSubsp} in this article). \end{re}

We conclude the section with a simple corollary.
\begin{co}\label{main}  Let $X$ be a complex surface and  ${\cal E}$ is a holomorphic rank $r$ bundle on $B\times X$ with $c_1({\cal E})\in p^*(H^2(B,\Z))+q^*(H^2(X,\Z))$
(so the  mixed  term  in the Künneth decomposition  of  $c_1({\cal E})$ vanishes). Suppose that there exists a line bundle ${\cal L}_0$ on $X$ such that 
$$h^i({\cal E}_b\otimes {\cal L}_0)=0\ \  \forall b\in B\ \forall i\in\{0,1,2\}\ ,\  $$
where ${\cal E}_b$ is the bundle on $X$ defined by the restriction  $\resto{{\cal E}}{\{b\}\times X}$. Then the map $[\delta_{\cal E}]:\Pic(X)\to\Pic(B)$ given by
$$[\delta_{\cal E}]([{\cal L}]:=[\lambda_{{\cal E}\otimes q^*({\cal L})}]$$
is constant  on every component  $\Pic^c(X)$ of $\Pic(X)$.
\end{co}
 
 \begin{proof} 
 
  Using Grauert semicontinuity  theorem and taking into account that $B$ is compact, we see that  for every $[{\cal L}]$ in a sufficiently small Zariski open neighborhood of $[{\cal O}_X]$ in $\Pic^0(X)$ one still has 
 $$h^i({\cal E}_b\otimes{\cal L}_0\otimes{\cal L})=0\ \forall b\in B\ \forall i\in\{0,1,2\}\ .$$
  Therefore $[\delta_{{\cal E}\otimes q^*({\cal L}_0)}]$ is constant on $\Pic^0(X)$, in particular the linearization $l_{{\cal E}\otimes q^*({\cal L}_0)}$ of its restriction to $\Pic^0_{\rm uf}(X)$ vanishes. 
Fix $c\in\NS(X)$ and let ${\cal T}$ be a holomorphic line bundle with Chern class $c$ on $X$.   Using  Corollary \ref{Co1} we see that under our assumptions
$$l_{{\cal E}\otimes q^*({\cal L}_0)}=l_{\cal E}=l_{{\cal E}\otimes q^*({\cal T})}\ ,$$
hence  $l_{{\cal E}\otimes q^*({\cal T})}=0$. The claim follows now by Remark \ref{ConstantSpeed}.
 
\end{proof}
 
\section{Compact subspaces of ${\cal M}^\st$}\label{CompSubsp}

\def\reg{\mathrm{reg}}

Let $X$ be a class VII surface with $b_2(X)>0$ endowed with a Gauduchon metric $g$, and let $E$ be a differentiable rank 2-bundle on $X$ with $c_2(E)=0$, $\det(E)=K$ (the underlying differentiable line bundle of  the canonical holomorphic line budnle ${\cal K}$). The fundamental object intervening in the our previous articles dedicated to existence to existence of curves on class VII surfaces is the moduli space ${\cal M}:={\cal M}^\pst_{\cal K}(E)$ of polystable holomorphic structures ${\cal E}$ on $E$ with $\det({\cal E})={\cal K}_X$ modulo the complex gauge group $\Gamma(X,\mathrm{SL}(E))$. This moduli space is endowed with the topology induced by the Kobayashi-Hitchin correspondence and is always compact\footnote{For $b_2(X)\leq 3$ the compactness of ${\cal M}$ follows directly from Uhlenbeck compactness theorem for moduli spaces of ASD connection \cite{Te2}. N. Buchdahl gave a proof for arbitrary $b_2$ \cite{Te4}. }. Its open subspace ${\cal M}^\st:={\cal M}^\st_{\cal K}(E)$ is a complex space of dimension $b_2(X)$ and the complement ${\cal R}:={\cal M}\setminus{\cal M}^\st$  (the subspace of split polystable bundles, or the space of reductions in ${\cal M}$) is a finite union of circles. We denote by ${\cal M}^\st_\reg$ the subspace of regular points in ${\cal M}^\st$, i.e. of points $[{\cal E}]\in{\cal M}^\st$ with $H^2({\cal E}nd_0({\cal E}))=0$. 
Excepting the case when $X$ belongs to a very special class of Kato surfaces we have  ${\cal M}^\st_\reg={\cal M}^\st$ for suitable Gauduchon metrics (see Proposition 1.3  and Lemma 2.3 in \cite{Te4}).

\subsection{Universal bundles} 

\begin{lm}\label{T} Let $Y\subset {\cal M}^\st$ be a compact complex subspace and let $c\in H^2(X,\Z)$. Then the set
$$\Pic^c(X)_Y:=\{[{\cal T}]\in\Pic^c(X)|\ h^0({\cal E}\otimes {\cal T})=h^2({\cal E}\otimes {\cal T})=0\ \forall [{\cal E}]\in Y\}\ $$
is Zariski open and non-empty in $\Pic^c(X)$.
\end{lm}

\begin{proof} The sets

$${\cal S}_0:=\{([{\cal T}],[{\cal E}])\in \Pic^c(X)\times Y|\ h^0({\cal E}\otimes {\cal T})>0\}\ ,\ 
$$
$${\cal S}_2:=\{([{\cal T}],[{\cal E}])\in \Pic^c(X)\times Y|\ h^2({\cal E}\otimes {\cal T})>0\}\ \ 
$$
are Zariski closed in $\Pic^c(X)\times Y$ by Grauert's semicontinuity theorem (applied locally, on  sufficiently small open sets  $Y'\subset Y$ such that a universal bundle on $Y'\times X$ exists).  Since $Y$ is compact, the projections $S_i$ of ${\cal S}_i$  on $\Pic^c(X)$ are Zariski closed by the Proper Mapping Theorem (\cite{GR} p. 213). We claim that $S_i\ne \Pic^c(X)$ for $i=0$, 2. Indeed, suppose for instance that $S_0=\Pic^c(X)$. Then for every $[{\cal T}]\in\Pic^c(X)$ there exists $[{\cal E}]\in Y$ such that $h^0({\cal E}\otimes {\cal T})>0$, in other words there exists a non-trivial morphism ${\cal T}^\vee\to {\cal E}$. But the degree map $\deg_g$ is surjective on   $\Pic^c(X)$ (because it is surjective  on $\Pic^0(X)$)  hence, choosing $[{\cal T}]\in\Pic^c(X)$   with $\deg_g({\cal T})\leq-\frac{1}{2}\deg_g({\cal K}_X)$, this will contradict the stability of ${\cal E}$. For the second set we use the same method taking into account that   
$$h^2({\cal E}\otimes {\cal T})=h^0({\cal K}\otimes  {\cal E}^\vee\otimes {\cal T}^\vee)=h^0({\cal E}\otimes{\cal T}^\vee)\ ,$$
 where the first equality follows   by Serre duality and the second is a consequence of the equality $\det({\cal E})={\cal K}_X$.

Since   $S_i\ne\Pic^c(X)$ and $S_i$ are Zariski closed, the union $S_0\cup S_2$ is a proper  Zariski closed subset of $\Pic^c(X)$ and the complement $\Pic^c(X)\setminus(S_0\cup S_2)$ is non-empty.   
 
\end{proof}

\begin{re}\label{chi}
For a holomorphic structure ${\cal E}$ on $E$ and a line bundle ${\cal T}$ on $X$ we have
$\chi({\cal E}\otimes {\cal T})=c_1({\cal T})^2$.
\end{re}
\begin{proof}
Choosing ${\cal E}={\cal O}_X\oplus {\cal K}_X$ we obtain by  Riemann-Roch theorem on surfaces
$$\chi({\cal E}\otimes{\cal T})=\chi({\cal T})+\chi({\cal K}_X\otimes{\cal T})=\chi({\cal T})+\chi({\cal T}^\vee)=c_1({\cal T})^2+2\chi({\cal O}_X)=c_1({\cal T})^2\ .$$
\end{proof}

The first important result of this section concerns the existence of universal bundles.
\begin{pr}\label{universal}
Let $Y\subset {\cal M}^\st_\reg$ be a compact complex subspace and $c\in H^2(X,\Z)$ such that $c^2=-1$. Then for every line bundle ${\cal T}$ on $X$ with $[{\cal T}]\in\Pic^c(X)_Y$ there exists an open neighborhood  ${\cal U}_{\cal T}$ of $Y$ in ${\cal M}^\st_\reg$ and   a universal bundle ${\cal F}^{\cal T}$ on ${\cal U}_{\cal T}\times X$, canonically associated with ${\cal T}$, with the following properties:
\begin{enumerate}[(i)]
\item The determinant line bundle of ${\cal F}^{\cal T}$ has the form 
$$\det({\cal F}^{\cal T})=p_{{\cal U}_{\cal T}}^*({\cal N}^{\cal T})\otimes p_X^*({\cal K}_X)$$
 for a line bundle ${\cal N}^{\cal T}$ on ${\cal U}_{\cal T}$.
\item There exists a canonical isomorphism 
$$R^1(p_{{\cal U}_{\cal T}})_*({\cal F}^{\cal T}\otimes p_X^*({\cal T}))={\cal O}_{{\cal U}_{\cal T}}\ .
$$

\end{enumerate}

 \end{pr}
 We will construct  the universal bundle ${\cal F}^{\cal T}$ using the  same construction method as in the proof of  Proposition 3.1   in \cite{Te4}. Before beginning the proof we recall first why, for an open set ${\cal U}\subset{\cal M}^\st$, the construction of a universal  bundle  on   ${\cal U}\times X$ might be obstructed. Understanding this obstruction will clarify  the method used in the proof, in particular the role of the line bundle ${\cal T}$ in removing this obstruction. 

A point $y\in {\cal U}$ is an equivalence class   of stable holomorphic structures on $E$ with determinant ${\cal K}_X$, modulo the gauge group $\Gamma(X,\mathrm{SL}(E))$. For two representatives ${\cal E}'$, ${\cal E}''$ of $y$ we can find a holomorphic isomorphism $f: {\cal E}'\to {\cal E}''$ with $\det(f)\equiv 1$, but $f$ is not unique; it is only well defined up sign, hence we have two such isomorphisms. For this reason in general it is not possible to define a holomorphic bundle on $Y\times X$  in a coherent way by selecting representatives ${\cal E}_y\in y$ for every $y\in {\cal U}$. The obstruction to the existence of a universal bundle on ${\cal U}\times X$ is a cohomology class in $H^2({\cal U},\Z_2)$.

\begin{proof} (i)  Since $[{\cal T}]\in\Pic^c(X)_Y$ (see Lemma \ref{T}) we have 
\begin{equation}\label{h0h2}
h^0({\cal E}\otimes{\cal T} )=h^2({\cal E}\otimes{\cal T})=0\ \forall [{\cal E}]\in Y\ ,
\end{equation}
Using Grauert's semicontinuity theorem  this implies the existence of a Zariski open neighborhood ${\cal U}_{\cal T}$ of $Y$ in the complex manifold ${\cal M}^\st_\reg$ such that
\begin{equation}\label{h0h2new}
h^0({\cal E}\otimes{\cal T})=h^2({\cal E}\otimes{\cal T})=0\ \forall [{\cal E}]\in {\cal U}_{\cal T}\  .
\end{equation}
On the other hand, by Remark \ref{chi},   $c_1({\cal T})^2=-1$ implies $\chi({\cal E}\otimes{\cal T})=-1$ for every holomorphic structure ${\cal E}$ on $E$.  Therefore,  for every holomorphic structure ${\cal E}$ with $[{\cal E}]\in {\cal U}$ we obtain a  well defined complex line 
$$\lg_{\cal E}:=H^1({\cal E}\otimes {\cal T})\ .$$
 For two holomorphic structures ${\cal E}'$, ${\cal E}''$ with $[{\cal E}']=[{\cal E}'']\in {\cal U}_{\cal T}$ we have this time  a {\it well-defined} holomorphic isomorphism 
${\cal E}'\otimes \lg_{{\cal E}'}^\vee\to {\cal E}''\otimes \lg_{{\cal E}''}^\vee$, because the two  holomorphic isomorphisms ${\cal E}'\rightrightarrows {\cal E}''\in \Gamma(X,\mathrm{SL}(E))$ obviously induce the same isomorphism ${\cal E}'\otimes \lg_{{\cal E}'}^\vee\to {\cal E}''\otimes \lg_{{\cal E}''}^\vee$. Therefore the family $({\cal E}\otimes\lg_{\cal E}^\vee)_{[{\cal E}]\in{\cal U}_{\cal T}}$ of holomorphic vector bundles   on $X$ defines in a coherent way a universal bundle ${\cal F}^{\cal T}$ on ${\cal U}_{\cal T}\times X$. The determinant line bundle of  ${\cal F}^{\cal T}$    is $p_{\cal U}^*({\cal N}^{\cal T})\otimes p_X^*({\cal K}_X)$, where ${\cal N}^{\cal T}$ is the holomorphic line bundle on ${\cal U}_{\cal T}$ defined by the family of lines $({\lg_{\cal E}^\vee}^{\otimes 2})_{[{\cal E}]\in{\cal U}_{\cal T}}$.
\\ \\
(ii) Fix a point $y=[{\cal E}]\in {\cal U}_{\cal T}$ and let ${\cal F}^{\cal T}_y$ be the bundle on $X$ defined by the restriction $\resto{{\cal F}^{{\cal T}}}{\{y\}\times X}$. We have natural identifications
$$H^1({\cal F}^{\cal T}_y\otimes {\cal T})=H^1({\cal E}\otimes \lg_{\cal E}^\vee\otimes {\cal T})=H^1({\cal E}\otimes H^1({\cal E}\otimes {\cal T})^\vee\otimes {\cal T})=$$
$$=H^1({\cal E}\otimes {\cal T})\otimes H^1({\cal E}\otimes {\cal T})^\vee=\C\ .
$$
By Grauert's local triviality and base change theorems it follows that the sheaf $R^1(p_{{\cal U}_{\cal T}})_*({\cal F}^{\cal T}\otimes p_X^*({\cal T}))$ is locally free of rank 1 and its fiber at $y\in {\cal U}_{\cal T}$ can be canonically identified with $\C$.
\end{proof}

The proof of the next result makes use of the variation formula and its corollaries proved in section \ref{variation}.
\begin{pr} \label{T0T1}
Let $c\in H^2(X,\Z)$ with $c^2=-1$, $Z$ be a connected, compact, complex manifold and $s:Z\to{\cal M}^\st_\reg$ a holomorphic map. For every   line bundle ${\cal T}$ on $X$ with $[{\cal T}]\in\Pic^c(X)_Y$ put ${\cal E}^{\cal T}:=(s\times\id_X)^*({\cal F}^{\cal T})$, where ${\cal F}^{\cal T}$ is the universal bundle associated with ${\cal T}$ and the compact subspace $Y:=\im(s)$ of ${\cal M}^\st_\reg$. Then for any two points $[{\cal T}_0]$, $[{\cal T}_1]\in\Pic^c(X)_Y$ one has
$$R^1(p_{Z})_*({\cal E}^{{\cal T}_0}\otimes p_X^*({\cal T}_1))\simeq {\cal O}_Z\ .
$$
\end{pr}
\begin{proof} For $[{\cal L}_0]\in\Pic^0(X)_Y$ we have
$$h^0({\cal E}\otimes {\cal L}_0)=h^2({\cal E}\otimes {\cal L}_0)=0=h^1({\cal E}\otimes {\cal L}_0)=0\ \forall [{\cal E}]\in Y\ ,
$$
where the first two equalities follow from the definition of $\Pic^0(X)_Y$ and the third from $\chi({\cal E}\otimes {\cal L}_0)=0$. Therefore
$$h^0({\cal E}^{{\cal T}_0}_z\otimes {\cal L}_0)=h^1({\cal E}^{{\cal T}_0}_z\otimes {\cal L}_0)=h^2({\cal E}^{{\cal T}_0}_z\otimes {\cal L}_0)=0\ \forall z\in Z\ .
$$
The determinant line bundle $\det({\cal E}^{{\cal T}_0})$ is $s^*({\cal N}_{{\cal T}_0})\otimes p_X^*({\cal K}_X)$, hence the  bundle ${\cal E}^{{\cal T}_0}$ satisfies the topological condition required in the hypothesis  Corollary \ref{main}. It follows by this Corollary that the map $[\delta_{{\cal E}^{{\cal T}_0}}]$ is constant on  every component of $\Pic(X)$ (not only on $\Pic^0(X)$), in particular on $\Pic^c(X)$. This gives a {\it holomorphic} isomorphism of  determinant line bundles
$$\lambda_{{\cal E}^{{\cal T}_0}\otimes p_X^*({\cal T}_0)}\simeq \lambda_{{\cal E}^{{\cal T}_0}\otimes p_X^*({\cal T}_1)}\ .
$$ 
 But 
$$\lambda_{{\cal E}^{{\cal T}_0}\otimes p_X^*({\cal T}_i)}=\left\{R^1(p_{Z})_*({\cal E}^{{\cal T}_0}\otimes p_X^*({\cal T}_i)\right\}^\vee \ ,$$
for $i=0$, 1, because the direct images $R^k(p_{Z})_*({\cal E}^{{\cal T}_0}\otimes p_X^*({\cal T}_i))$ vanish for $k=0$, 2.  Therefore 
 $$R^1(p_{Z})_*({\cal E}^{{\cal T}_0}\otimes p_X^*({\cal T}_1))\simeq R^1(p_{Z})_*({\cal E}^{{\cal T}_0}\otimes p_X^*({\cal T}_0))\simeq$$
 $$
s^*\big(R^1(p_{{\cal U}_{{\cal T}_0}})_*({\cal F}^{{\cal T}_0}\otimes p_X^*({\cal T}_0))\big) \simeq s^*({\cal O}_{{\cal U}_{{\cal T}_0}})= {\cal O}_Z\ ,
 $$
 where the second isomorphism is given by the  the base change property for direct images which holds in our case (see Proposition 4.10   \cite{BS}) and the third by Proposition \ref{universal} (ii).
\end{proof}

 \subsection{A non-existence result}

We recall (see \cite{Te3}, \cite{Te4}) that, by definition, the canonical extension   of $X$ is the essentially unique non-trivial extension of the form
\begin{equation}\label{Anew}
0\to {\cal K}_X\textmap{i_0}  {\cal A}\textmap{p_0} {\cal O}_X\to 0\ .
\end{equation}

  The bundle ${\cal A}$ defined in this way plays a crucial role in our program to prove existence of curves on class VII surfaces \cite{Te4}. The point is that there is no obvious way to identify explicitly the small deformations of ${\cal A}$ or even  
 to decide whether it admits  small deformations which are filtrable.
 
  From now on we will always suppose that the pair $(X,g)$ was chosen such that ${\cal A}$ is stable. In this case it defines a regular point in the moduli space ${\cal M}^\st$ (see  the proof of Proposition 1.3 in \cite{Te4} p. 1758).  We put $a:=[{\cal A}]\in{\cal M}^\st_\reg$.

\begin{thry}\label{a(Z)}  Let $Z$ be a connected, compact, complex manifold and $s:Z\to{\cal M}^\st_\reg$ a holomorphic map such that
\begin{enumerate}[(i)]
\item $a\in\im(s)$ and $s^{-1}(a)$ is a proper analytic set of $Z$.
\item The non-filtrable locus $Z_{\rm nf}:=\{z\in Z|\ s(z)\hbox{ is non-filtrable}\}$ in $Z$ has non-empty interior.
\end{enumerate}
Then $a(Z)>0$.
\end{thry}
\begin{proof}

 Fix $c\in H^2(X,\Z)$ such that
 $$c^2=-1\ ,\  c_1({\cal K}_X)\ c >0\ .$$
  In other words one has $c=-e$ mod $\Tors$, where $e$ is an element of the Donaldson  basis of $H^2(X,\Z)/\Tors$ (see \cite{Te4}, \cite{Te6}). Put $Y:=\im(s)$, fix a line bundle ${\cal T}_0$ on $X$ with $[{\cal T}_0]\in\Pic^c(X)_Y$ and put
$${\cal E}^0:=(s\times\id_X)^*({\cal F}^{{\cal T}_0})\ .
$$
${\cal E}^0$ is a universal bundle for $s$, i.e. one has ${\cal E}^0_z\in s(z)$ for any $z\in Z$.

Lemma \ref{pLemma} below states that $(p_X)_*({\cal E}^0)\ne 0$. The claim follows now from Proposition \ref{filt-non-filt} proved in the Appendix.
\end{proof}
\begin{lm}\label{pLemma}
In the conditions of Theorem \ref{a(Z)} and with the notations introduced above, one has $(p_X)_*({\cal E}^0)\ne 0$.
\end{lm}

\begin{proof} We will give a proof by reductio ad absurdum. Supposing that $(p_X)_*({\cal E}^0)= 0$, we will come to a contradiction in the following way: For a second line bundle ${\cal T}$ with  $[{\cal T}]\in\Pic^c(X)_Y$ we put
$$\TE^0:={\cal E}^0\otimes p_X^*({\cal T})\ .$$

We will compute the cohomology space $H^1(\TE^0)$ using the Leray spectral sequences associated with the two projections $p_Z$, $p_X$. Using the the Leray spectral sequence  associated with $p_Z$ we will obtain $h^1(\TE^0)=1$, and using the Leray spectral sequence  associated with $p_X$ for convenient $[{\cal T}]\in\Pic^c(X)_Y$ we will obtain $h^1(\TE^0)=0$.
We divide the proof in several steps which will point out clearly the ideas of the proof.
\\ \\
{\it Step 1.} Proving that $h^1(\TE^0)=1$ using the Leray spectral sequence  associated with $p_Z$.
\\
   
For $z\in Z_{\rm nf}$ (which by assumption is non-empty) the bundle ${\cal E}_z^0$ is non-filtrable, hence one has $h^0({\cal E}_z^0\otimes {\cal T})=0$. By Grauert's semicontinuity and local-triviality theorems it follows that the sheaf $(p_Z)_*(\TE^0)$ vanishes generically, hence (being torsion free) must vanish.
 On the other hand, by  Proposition \ref{T0T1}
 $$R^1(p_Z)_*(\TE^0)=R^1(p_Z)_*({\cal E}^0\otimes p_X^*({\cal T}))\simeq{\cal O}_Z\ . 
 $$
The Leray spectral sequence  associated with the projection $p_Z$ yields the exact sequence
$$0\to H^1((p_Z)_*(\TE^0))\to H^1(\TE)\to H^0(R^1(p_Z)_*(\TE^0))\to H^2((p_Z)_*(\TE^0))\to H^2(\TE^0),$$
which proves that in our case the canonical morphism $H^1(\TE)\to H^0(R^1(p_Z)_*(\TE^0))$ is an isomorphism and $h^1(\TE^0)=h^0({\cal O}_Z)=1$.
\\ \\
{\it Step 2.} Proving that, under the assumption $(p_X)_*({\cal E}^0)= 0$,   for a suitable     $[{\cal T}]\in\Pic^c(X)_Y$, the Leray spectral sequence  associated with $p_X$ yields $h^1(\TE^0)=0$ .
\\  
 
 The statement follows from Proposition 3.5 in \cite{Te6}. We recall the argument for completeness. By the projection formula we obtain $(p_X)_*(\TE^0)=(p_X)_*({\cal E}^0)\otimes {\cal T}=0$, hence  as in Step 1. we obtain an isomorphism
 $$H^1(\TE)\textmap{\simeq} H^0(R^1(p_X)_*(\TE^0))\ .
 $$
Therefore it suffices to prove that, for convenient $[{\cal T}]\in\Pic^c(X)_Y$, we have 
$$H^0(R^1(p_X)_*(\TE^0))=0\ .$$

We use the short exact sequence
 $$0\to \Tors(R^1(p_X)_*(\TE^0))\to R^1(p_X)_*(\TE^0)\to \TS^0\to 0
 $$
 where $\TS^0$ is the quotient of $R^1(p_X)_*(\TE^0)$ by its torsion subsheaf, hence it is a torsion free sheaf. We will prove that, for convenient ${\cal T}$, the  spaces of sections of the  two extreme sheaves in this exact sequence vanish. 
 \\ \\
 {\it Step 2a.} Proving that for convenient $[{\cal T}]\in\Pic^c(X)_Y$ one has 
 $$H^0(\Tors(R^1(p_X)_*(\TE^0)))=0\ .$$
 
More precisely we have  $H^0(\Tors(R^1(p_X)_*(\TE^0)))=0$ for any $[{\cal T}]\in\Pic^c(X)_Y$ except for at most one point. This follows by Proposition 3.5 (i) in \cite{Te6}.  The proof is based on the general description of the torsion subsheaf of the first direct image obtained in this article (Theorem 2.5 in \cite{Te6}). Note that for this vanishing result we only need the condition $a\in\im(s)$, hence the existence of non-filtrable points in $\im(s)$ is not needed here. On the other hand for this result one needs the special choice of the sign of $c$ ($c=-e$ mod $\Tors$, where $e$ is an element of the Donaldson basis of $X$);  all other arguments in this proof  are also valid for $c=e$.
 \\ \\
 {\it Step 2b.} Proving that for convenient $[{\cal T}]\in\Pic^c(X)_Y$ one also has $H^0(\TS^0)=0$.
 \\
 
 By the projection formula one has
 $$R^1(p_X)_*(\TE^0)= R^1(p_X)_*( {\cal E}^0\otimes p_X^*({\cal T}))=R^1(p_X)_*({\cal E}^0)\otimes {\cal T}\ . 
 $$
 Since ${\cal T}$ is locally free, this gives an isomorphism 
 $$\TS^0\simeq {\cal S}^0\otimes {\cal T}\ ,$$
  where ${\cal S}^0$ is the quotient of $R^1(p_X)_*({\cal E}^0)$ by its torsion. Since ${\cal S}^0$ is torsion  free, it  can be identified with a subsheaf of its bidual ${\cal B}^0:=[{\cal S}^0]^{\vee\vee}$.  Using a Bochner type vanishing theorem, we will show that, if the degree of ${\cal T}$ is sufficiently negative, one has $H^0({\cal B}^0\otimes {\cal T})=0$, which will complete the proof because (since ${\cal T}$ is locally free)  $\TS^0\simeq {\cal S}^0\otimes {\cal T}$ is a subsheaf of ${\cal B}^0\otimes {\cal T}$.
  
 ${\cal B}^0$ is a reflexive sheaf on a surface, hence is locally free. Fix a Hermitian metric $h^0$ on (the associated holomorphic vector bundle of) ${\cal B}^0$.  Let $h_{\cal T}$ be  a Hermitian-Einstein metric  on ${\cal T}$ (see Corollary 2.1.6 \cite{LT}). This implies 
  $$i\Lambda_g F_{{\cal T},h_{{\cal T}}}=c_{\cal T}\ ,$$
   where $c_{\cal T}$ is the Einstein constant of the corresponding Chern connection and depends on the pair $(g,{\cal T})$. The Chern connection of the tensor product metric $h^0\otimes h_{\cal T}$ on ${\cal B}^0\otimes {\cal T}$ is the tensor product of the Chern connections of the factors. Using formula (5.13) in \cite{K}, we get
   $$i\Lambda_g F_{{\cal B}^0\otimes {\cal T},h^0\otimes h_{\cal T}}=i\Lambda_gF_{{\cal B}^0,h^0}\otimes\id_{{\cal T}}+(i\Lambda_g F_{{\cal T},h_{{\cal T}}})\id_{{\cal B}^0\otimes{\cal T}}=$$
   $$=i\Lambda_gF_{{\cal B}^0,h^0}\otimes\id_{{\cal T}}+c_{\cal T}\id_{{\cal B}^0\otimes{\cal T}}\ .
   $$
   This formula shows that the Hermitian endomorphism $i\Lambda_g F_{{\cal B}^0\otimes {\cal T},h^0\otimes h_{\cal T}}$ of ${\cal B}^0\otimes {\cal T}$ is negative definite when $c_{\cal T}$ is sufficiently negative.  By a well-known Bochner type vanishing theorem (see Theorem 1.9 p. 52 in \cite{K}) it follows that $H^0({\cal B}^0\otimes {\cal T})=0$ when $c_{\cal T}$ is sufficiently negative. But the Einstein constant $c_{\cal T}$ is proportional with $\deg_g({\cal T})$ (see Proposition 2.18 in \cite{LT}) and the degree map $\deg_g$ is surjective on $\Pic^c(X)$. The Zariski open set $\Pic^c(X)_Y$ is the complement of a closed, discrete set of points in $\Pic^c(X)\simeq\C^*$. Therefore there exists infinitely many points $[{\cal T}]\in \Pic^c(X)_Y$ with arbitrary negative degree, hence infinitely many points $[{\cal T}]\in \Pic^c(X)_Y$ for which $H^0({\cal B}^0\otimes {\cal T})=0$ (which implies $H^0(\TS^0)=0$).  Therefore, recalling that in {\it Step 2b)} we must avoid at most one point of $\Pic^c(X)$, we can find $[{\cal T}]\in \Pic^c(X)_Y$ such that simultaneously 
$$H^0(\TS^0)=0\ ,\ H^0(\Tors(R^1(p_X)_*(\TE^0)))=0\ .$$
Therefore $H^0(R^1(p_X)_*(\TE^0))=0$. This implies $h^1(\TE^0)=0$, which contradicts the equality $h^1(\TE^0)=1$ obtained above.
 \end{proof}
  
We can prove now 
\begin{thry} \label{noY} There does not exist any  irreducible reduced compact complex  subspace $Y\subset{\cal M}^\st_{\rm reg}$
such that
\begin{enumerate}[(i)]
\item $\dim(Y)\geq 1$ and $Y$ contains the point $a$.
\item \label{neighb} There exists an open neighborhood $Y_a\subset Y$ of $a$ in $Y$  such that $Y_a\setminus\{a\}$ contains only non-filtrable points.
 
\end{enumerate} 
\end{thry}
\begin{proof}
We will prove the statement by induction with respect to $n:=\dim(Y)$ starting with $n=1$. Let $\delta:Z\to Y$ be a desingularization of $Y$ and put $s:=\iota_Y\circ \delta$, where $\iota_Y:Y\to{\cal M}^\st_\reg$ is the embedding morphism. The pair $(Z,s)$ satisfies the assumptions  of Theorem \ref{a(Z)}. 

The image of $Z$ contains both filtrable and non-filtrable bundles, hence for $n=1$ the statement follows from Corollary 5.3 in \cite{Te2}. If $n\geq 2$  we make use of Theorem \ref{a(Z)}, which gives  $a(Z)>0$  hence, since $\delta$ is bimeromorphic, $a(Y)>0$. 

But any irreducible, reduced complex space $Y$ with $a(Y)>0$ is ``covered by divisors" (see \cite{CP} p. 331). A simple proof of this statement can be obtained as follows: consider a normalization $n:\hat Y\to Y$ and a non-constant meromorphic function on $\hat Y$. Then $\hat Y$ decomposes as 
$$Y={\cal P}(f)\union \left(\union_{a\in\C} {\cal P}\left(\frac{1}{f-a}\right)\right)\ ,$$
where ${\cal P}(f)$ denotes the pole variety of $f$ (\cite{Fi} p. 174, 175). Using the invariance of the dimension under finite maps (see \cite{Fi} Theorem p. 133), we obtain a 1 parameter family of pure 1-codimensional analytic subsets of $Y$ which covers Y.

Therefore, coming back to our subspace $Y\subset {\cal M}^\st$ satisfying (i) and (ii), there exists a pure 1-codimensional analytic subset  $Y'\subset Y$    containing $a$.  This analytic subset,  which can be supposed irreducible, also satisfies the assumptions (i), (ii), so assuming that the claim holds for $n-1$,  the existence of a subspace of dimension $n$ satisfying (i), (ii) leads to a contradiction.
\end{proof}

\subsection{Existence of a cycle on curves on class VII surfaces with $b_2=2$ }

In \cite{Te4} we developed a program for proving existence of a cycle of curves on minimal class VII surfaces and, using this program, we proved than any minimal class VII surface with $b_2(X)=2$ has a cycle of curves. We explain briefly   our program in the general case ($b_2(X)$ arbitrary), and then we show how Theorem \ref{noY} can be used to complete the proof in the case $b_2(X)=2$ in a new way. The case $b_2(X)>2$ will be discussed in a forthcoming article.

Let $X$ be a minimal class VII surface with $b:=b_2(X)>0$, and let $(e_1,\dots,e_b)$ be a Donaldson basis of $H^2(X,\Z)/\Tors$, i.e., a basis with the properties:
$$e_i e_j=-\delta_{ij}\ ,\ c_1({\cal K}_X)=\sum_{i=1}^b e_i\ \mod \Tors\ .
$$
(see \cite{Te3}). For every index $I\subset \Ig:=\{1,\dots,b\}$ we put  $e_I:=\sum_{i\in I}e_i$. As in the previous sections, let $E$ be a differentiable rank 2 Hermitian bundle on $X$ with $c_2(E)=0$ and $\det(E)=K_X$ (the underlying differentiable line bundle of ${\cal K}_X$). The starting point of our program is the following simple remark which gives a complete classification of all filtrable points in our moduli space (see Proposition 1.1. in \cite{Te4}): If ${\cal E}$ is a holomorphic structure on $E$, ${\cal L}$ a holomorphic line bundle and $j:{\cal L}\to{\cal E}$ a sheaf monomorphism  with torsion free quotient, then $j$ is a bundle embedding (equivalently ${\cal E}/j({\cal L})$ is a locally free) and there exists $I\subset\Ig$ such that $c_1({\cal L})=e_I$.    An extension of the form 
$$0\to {\cal L}\to {\cal E}\to {\cal K}\otimes {\cal L}^\vee\to 0
$$
with $c_1({\cal L})=e_I$ will be called extension of type $I$. For instance, by definition, the canonical extension ${\cal A}$ is an extension of type $\Ig$  (see (\ref{Anew})). If $X$ is not an Enoki surface (which we will assume) there exists only finitely many extensions of type $\Ig$ (see the proof of Lemma 2.1 in \cite{Te4}).

Our general strategy to prove existence of a cycle on minimal class VII surfaces with $b_2>0$ is based on the following simple Remark (see Proposition 1.4 in \cite{Te4}):
\begin{pr}\label{diffK} If the middle term ${\cal A}$ of the canonical extension can be written as an extension in a different way (with a different kernel), then $X$ has a cycle of curves.
\end{pr}

The proof of this proposition is not difficult: one shows that if ${\cal N}$ is a holomorphic line bundle on  $X$ and $i:{\cal N}\to{\cal A}$  is a bundle embedding with $\im(i)\ne \im(i_0)$, then the image of $p_0\circ i$ in ${\cal O}_X$ is the ideal sheaf of a non-empty effective divisor $C\subset X$ with dualizing sheaf $\omega_C={\cal O}_C$.\\

The existence of a cycle on $X$ is proved by reductio  ad absurdum.
We fix a Gauduchon metric $g$ on $X$ with $\deg_g({\cal K}_X)<0$. This is always possible when $X$ is minimal. This is stated in Proposition 3.7 in \cite{Te3} and Lemma 2.3 in  \cite{Te4}, but this result is due to N. Buchdahl and is based on Buchdahl positivity criterion \cite{Bu}. For such a Gauduchon  metric the canonical kernel ${\cal K}_X$ of the canonical extension does not destabilize ${\cal A}$, hence, by Proposition \ref{diffK}, if $X$ has no cycle, ${\cal A}$ will be stable. Using the notations introduced at the beginning of this section we obtain a point $a:=[{\cal A}]\in{\cal M}^\st$.  Supposing that $X$ is not an Enoki surface, for such a metric we have  ${\cal M}^\st={\cal M}^\st_{\rm reg}$ (see Proposition 1.3 \cite{Te4}).

  The subset of ${\cal M}^\st$ consisting of points represented by an extension of type $I$  will be denoted by ${\cal M}^\st_I$. The closure   $\bar{\cal M}_\emptyset$ of  ${\cal M}^\st_\emptyset$ in ${\cal M}^\st$ is always closed an open (so a union of connected components) in ${\cal M}^\st$. Moreover, using the compacity of ${\cal M}:={\cal M}^\pst_{\cal K}(E)$, one can prove that the complement ${\cal M}^\st\setminus \bar{\cal M}_\emptyset$ is compact. The point is that the closure of ${\cal M}_\emptyset$ in the compact space ${\cal M}$ is open and contains the whole space of reductions ${\cal R}={\cal M}\setminus {\cal M}^\st$. In other words $\bar{\cal M}_\emptyset$ contains all the {\it ends} of ${\cal M}^\st$.

 $\bar{\cal M}_\emptyset$  can be described explicitly in the case $b_2=2$. The result is very simple (see \cite{Te4}): for every class $c\in\Tors(H^2(X,\Z))$ consider a  ruled surface $\Pi_c$ obtained from $D\times\P^1$ by applying a (possibly empty) sequence of blow ups above the origin of the disk $D$. Therefore the fiber $T_c$ of $\Pi_c$ over $0\in D$ is a tree  of rational curves. One can construct a biholomorphism
  $$\Phi: \union_{c\in\Tors} (\Pi_c\setminus R_c)\textmap{\simeq} \bar{\cal M}_\emptyset\ ,
  $$
  where $R_c$ is a circle contained in an irreducible component $E_c$ of $T_c$. For a generic Gauduchon metric this circle does not intersect the other irreducible components of the tree $T_c$. One has
  $$\Phi(\union_{c\in\Tors} (\Pi_c\setminus T_c))={\cal M}^\st_\emptyset\ ,\ \Phi(\union_{c\in\Tors} (E_c\setminus R_c))={\cal M}^\st_{\{1\}}\cup {\cal M}^\st_{\{2\}}\cup {\cal B}\ ,
  $$  
where  ${\cal B}$ is a  subset of the set of {\it twisted reductions} in the moduli space. By definition, a point $[{\cal E}]\in{\cal M}^\st$  is a twisted reduction if   ${\cal {\cal E}}\simeq{\cal E}\otimes {\cal L}_0$, where $[{\cal L}_0]$ is a non-trivial square root of $[{\cal O}_X]$.  The  bull-back  of such a point to a suitable double cover of $X$ is   split polystable. The set of twisted reductions in  ${\cal M}^\st$ is finite.  

Using all these geometric properties and the non-existence Theorem \ref{noY} one can now easily  prove that the non-existence of a cycle on $X$ leads to a contradiction. Indeed, using Proposition \ref{diffK} it follows that the hypothesis ``$X$ has no cycle" implies that the point $a=[{\cal A}]$ belongs neither to ${\cal M}^\st_\emptyset$, nor to ${\cal M}^\st_{\{1\}}\cup {\cal M}^\st_{\{2\}}$, nor to the set ${\cal B}$.  Therefore ${\cal A}$ belongs to either  an irreducible component $F_c\ne E_c$ of $T_c$ or to a connected component $Y$  of  ${\cal M}^\st\setminus \bar{\cal M}_\emptyset$ (which is compact). But, since ${\cal M}^\st_\Ig$ is finite, $F_c$ and $Y$  can contain  only finitely many filtrable points, hence both cases would contradict Theorem \ref{noY}.
\begin{re} Theorem \ref{noY} concerns minimal class VII surfaces with arbitrary $b_2>0$. In order to generalize our method it remains to give a sufficiently explicit description of the subspace $\bar{\cal M}_\emptyset$ of ${\cal M}^\st$ for  $b_2>2$. More precisely it suffices to    prove that, as in the case $b_2=2$ explained above,   this subspace contains the union $\cup_{I\ne\Ig}{\cal M}^\st_I$, and has a stratification, each stratum being either a space of extensions, or a space of twisted reductions, or a compact complex space whose generic point is non-filtrable.  \end{re}

Note that, by a result of Nakamura \cite{Na2},  any minimal class VII surface with $b_2>0$ which contains a cycle of curves is a degeneration of a one-parameter family of blown up primary Hopf surfaces,  hence proving the existence of a cycle in the general case would complete the classification of class VII surfaces up to deformation equivalence.

 \section{Appendix}
 
 \subsection{1-parameter families of connections and connections in temporal gauge}
 
 Let  $X$ be a connected differentiable manifold, $L$   a Hermitian line bundle on   $X$ and $P$ the associated principal $S^1$-bundle. Let $\Ag=(A_t)_{t\in \R}$ be smooth 1-parameter family of connections on $L$ (or, equivalently, on the principal bundle $P$). Such a family can be regarded as a smooth map
$\Ag:\R\to {\cal A}(L)={\cal A}(P)$. The latter space is an affine space with model vector space $iA^1(X)$,  hence one can define the  derivative $\dot\Ag(t) \in iA^1(X)$. 
Let  $\Ag=(A_t)_{t\in \R}$ be a family as above; for any $t\in \R$ denote by $\theta_t\in iA^1(P)$ the connection form of $A_t$.   The product  $\tilde L:=\R\times L$ ($\tilde P:=\R\times P$) will be regarded as a Hermitian line bundle (respectively principal $S^1$-bundle)  on $\R\times X$. We denote by $\theta\in iA^1(\tilde P)$ the   1-form which coincides with $\theta_t$ on $\{t\}\times P$ and vanishes in the $\R$-direction, i.e., on the subbundle $p_\R^*(T_\R)$ of $T_{\R\times P}$. This form is the connection form of a connection $\tilde A$ on $\tilde P$ whose horizontal space at $(t,p)\in\tilde P$ is $T_t(\R)\oplus H_p^{t}$, where $H^t\subset T_{P}:=\ker(\theta_t)$ denotes the horizontal distribution of the connection $A_t$.   Using the terminology used in gauge theory (see for instance \cite{D} p. 15) $\tilde A$ is a connection ``in temporal gauge" on the product bundle $\tilde P:=\R\times P$, because  the $dt$-component of its connection form vanishes.

The curvature form $F_{\tilde A}$ of the connection $\tilde A$ is an imaginary 2-form on $\R\times X$ whose components 
with respect to the bundle decomposition
$$\Lambda^2_{\R\times X}= (p_\R^*(\Lambda^1_\R)\otimes p_X^*(\Lambda^1_X))\oplus  p_X^*(\Lambda^2_X)\ .$$
are 
$$ F_{\tilde A}^{\R X}\left(\frac{d}{dt},\xi\right)=\dot \Ag(t)(\xi)\ ,\ \resto{F_{\tilde A}^{XX}}{\{t\}\times X}=F_{A_t}\ .  
$$
Therefore the mixed term of the curvature of a connection $\tilde A$ in temporal gauge  has an important  geometric interpretation: it determines the derivative of the 1-parameter family of connections used in the construction of $\tilde A$. 

Conversely, let $Q$ be an arbitrary principal $S^1$-bundle on  the product $\R\times X$ and $B$ a connection on $Q$. Denote by $P$ the $S^1$-bundle on $X$ defined by the restriction $\resto{Q}{\{0\}\times X}$. Using parallel transport with respect to the connection $B$ along the lines $\R\times\{x\}$ we obtain an identification 
$$f_{B}:P\times\R=:\tilde P\textmap{\simeq}   Q\ ,
$$
and the pull-back of $B$ to $P\times\R$ is just the connection  $\tilde A$ associated with the family of connections $\Ag=(A_t)_{t\in\R}$, where $A_t=\resto{\iota_{B}^*(B)}{\{t\}\times P}$. An isomorphism $f_{B}$ constructed as above (using parallel transport in the $\R$-directions) is called a temporal gauge for the connection $B$ (because it identifies $B$ with a connection in temporal gauge). We have proved
\begin{pr} \label{family} Let $Q$ be a principal $S^1$-bundle   on $\R\times X$ and  let $P$ be the   principal  $S^1$-bundle on $X$ defined by the restriction $\resto{Q}{\{0\}\times X}$. Then  for any connection $B$ on  $Q$ on $\R\times X$ there exists
\begin{enumerate}
\item   a bundle identification $f_{B}:P\times\R=:\tilde P\textmap{\simeq} Q$,
\item a  smooth family $\Ag=(A_t)_{t\in \R}$ of connections on $P$, 
\end{enumerate}
such that $f_B^*(B)$ is induced by the family $\Ag$ via the construction above. In particular one has
$$\dot\Ag(t)=\iota_{\frac{d}{dt}}F_{B}^{\R X}= \iota_{\frac{d}{dt}}F_{B}\hbox{ (regarded as a 1-form on }X)\ .
$$
\end{pr}
\vspace{4mm}

A smooth 1-parameter family of  Hermitian connections of type (1,1) on a compact complex manifold $X$ defines a path in  $\Pic(X)$. The velocity of this path can be computed using

\begin{lm} \label{path-Pic} Suppose that $X$ is a compact complex manifold, $L$ a Hermitian line bundle on $X$  and $\Ag=(A_t)_{t\in \R}$ is  a smooth family of Hermitian connections of type (1,1) on $L$. Let ${\cal L}_t$ be the holomorphic line bundle defined by the Dolbeault operator $\bar\partial_{A_t}$ on $L$ and $[{\cal L}_t]$ the corresponding point in $\Pic(X)$. Then
$$\frac{d}{dt} [{\cal L}_t]=[\dot\Ag(t)^{0,1}]\in H^1(X,{\cal O}_X)\ .
$$
\end{lm}
\begin{proof}
The claim is an easy consequence of the gauge theoretical interpretation of $\Pic(X)$. We recall that a semiconnection on $L$ is a  $\C$-linear first order operator $\delta:A^0(L)\to A^{0,1}(L)$ satisfying the usual Leibniz rule with respect to multiplication with smooth complex functions \cite{LT}. A semiconnection $\delta$ on $L$ is called integrable (or a Dolbeault operator on $L$)  if the 0-order operator $F_\delta^{0,2}:=\delta\circ\delta \in A^{0,2}(X)$  vanishes. An integrable semiconnection $\delta$ defines a holomorphic structure ${\cal L}_\delta$ on $L$ whose associated locally free sheaf is given by
$${\cal L}_\delta(U):=\{\sigma\in\Gamma(U,L)|\ \delta\sigma=0\}$$
for open sets $U\subset X$.  
Putting $c:=c_1(L)$, the connected component  $\Pic^c(X)$ of $\Pic(X)$ can be identified \cite{LT} with the quotient
\begin{equation}
{\cal M}(L):=\qmod{{\cal H}(L)}{{\cal G}^\C}\ ,
\end{equation}
where ${\cal H}(L)$ is the space of integrable semiconnections  
on $X$ and ${\cal G}^\C$ is the complex gauge group acting on ${\cal H}(L)$ by
$$ f\cdot \delta= f\circ \delta\circ f^{-1}=\delta-(\bar\partial f) f^{-1}\ .
$$

Using suitable Sobolev completions and local slices for the ${\cal G}^\C$-action,  the quotient ${{\cal H}(L)}/{{\cal G}^\C}$ can be endowed with a natural structure of a complex manifold. The tangent space $T_\delta({\cal H})$ at ${\cal H}$ in $\delta\in{\cal H}$ is the space $Z^{0,1}_{\bar\partial}$ of $\bar\partial$-closed (0,1)-forms, and  the tangent space $T_\delta({\cal G}^\C\cdot\delta)$ at the orbit ${\cal G}^\C\cdot\delta$ is the subspace $B^{0,1}_{\bar\partial}$ of $\bar\partial$-exact (0,1)-forms. Therefore the tangent space $T_{[\delta]}({\cal M}(L))$ can be naturally identified with $H^1({\cal O}_X)$ via the Dolbeault isomorphism.  The point is that the obtained isomorphism $T_{[\delta]}({\cal M}(L))\to H^1({\cal O}_X)$ is precisely the differential of the identification ${\cal M}(L)\textmap{\simeq}\Pic^c(X)$.

To complete the proof it suffices to note that, by definition, the line bundle ${\cal L}_t$ coincides with ${\cal L}_{\bar\partial_{A_t}}$, hence
$$\frac{d}{dt} [{\cal L}_t]=\left[\frac{d}{dt}(\bar\partial_{A_t})\right]=\left[\left(\frac{d}{dt}  A_t\right)^{0,1}\right]=[\dot\Ag(t)^{0,1}]\ .
$$
\end{proof}

Combining Proposition \ref{family} with Lemma \ref{path-Pic} we obtain the following differential geometric method to compute the differential of a holomorphic map induced by a family of holomorphic line bundles. In order to state this result we introduce the following notation: for two  differentiable manifolds $Y$, $X$,  a 2-form $F$  on $Y\times X$,  a point $y\in Y$ and a tangent vector $v\in T_yY$ we denote by $\iota_v F$ the 1-form on $X$ defined by
$$\iota_v F(\xi):=F(v,\xi)\ \forall x\in X\ \forall \xi\in T_x(X)\ .
$$
\begin{pr}\label{diff-curv}
Let $X$, $Y$  be  complex manifolds with $X$ compact, ${\cal L}$ a holomorphic line bundle on $Y\times X$ and $\varphi:Y\to\Pic(X)$ the induced holomorphic map. Fix a Hermitian metric $h$ on ${\cal L}$ and let $F_{{\cal L},h}\in iA^{1,1}(Y\times X)$ be the curvature of the corresponding Chern connection $A_{{\cal L},h}$.
For any point $y\in Y$ and tangent vector $v\in T_y^\R Y$ one has (via the Dolbeault isomorphism)
$$d\varphi(v)=[\iota_{v}(F_{{\cal L},h})^{0,1}]\in H^1(X,{\cal O}_X)\ .$$
 \end{pr}
 \begin{proof} For $y\in Y$ we denote by ${\cal L}_y$ the holomorphic line bundle on $X$ defined by the restriction of ${\cal L}$ to $\{y\}\times X$ and by $h_y$ the Hermitian metric on ${\cal L}_y$ induced by $h$. Therefore one has
 $$\varphi(y)=[{\cal L}_y]\ \forall y\in Y\ .
 $$
 Let $\gamma:\R\to Y$ be any differentiable path such that $\dot\gamma(0)=v$ and put $f:=\gamma\times\id_X:\R\times X\to Y\times X$.  The pull-back connection $f^*(A_{{\cal L},h})$ is a connection on the Hermitian line bundle $f^*({\cal L},h)$ on $\R\times X$ hence, by Proposition \ref{family}  it defines a smooth family $\Ag=(A_t)_{t\in\R}$ of Hermitian connections on a fixed Hermitian line bundle $(L_0,h_0)$ on $X$. The point is that the explicit construction of this family gives for every $t\in\R$ a Hermitian isomorphism  
 $$f_t:(L_0,h_0)\to ({\cal L}_{\gamma(t)},h_{\gamma(t)})$$ 
 such that $(f_t)_*(A_t)$ coincides with the Chern connection $A_{\gamma(t)}$ of   $({\cal L}_{\gamma(t)},h_{\gamma(t)})$. In particular the holomorphic structure defined by the Dolbeault operator (integrable semiconnection) $\bar\partial_{A_t}$ on $L_0$ is isomorphic to 
 ${\cal L}_{\gamma(t)}$. It suffices to apply  Lemma \ref{path-Pic} to the family $\Ag$  using the formula for $\dot\Ag(t)$ given by Proposition \ref{family}.
 \end{proof}

 Let $Y$, $X$ be arbitrary differentiable manifolds and $\psi$ a  closed 2-form on $Y\times X$. We denote by $\psi^X$ the map $\psi^X:Y\to Z^2(X)$ obtained by restricting $\psi$ to the fibers $\{y\}\times X\simeq X$.
 \begin{re} \label{closed1} Let $y\in Y$ and $v\in T_yY$. Suppose that $\psi$ is closed and $d(\psi^X)(v)=0$. Then the 1-form $\iota_{v}\psi$ on $X$ is closed.
\end{re}

Indeed, choose tangent fields $\alpha$, $\beta\in {\cal X}(X)$ with $[\alpha,\beta]=0$ and   $\gamma\in{\cal X}(Y)$ with $\gamma_y=v$. Denoting by $\tilde\alpha$, $\tilde\beta$, $\tilde\gamma$ the corresponding fields on $Y\times X$ our assumptions imply
$$0=d\psi(\tilde \alpha,\tilde \beta,\tilde \gamma)=\tilde\alpha(\psi(\tilde\beta,\tilde\gamma))-\tilde\beta(\psi(\tilde\alpha,\tilde\gamma))+\tilde\gamma(\psi(\tilde\alpha,\tilde\beta))
$$
$$0=\{d(\psi^X)(v)\}(\alpha,\beta):=v(\psi((\tilde \alpha,\tilde \beta)) \ .
$$
Restricting the first formula to $\{y\}\times X$ we get 
$$\alpha(\iota_v\psi(\beta))-\beta(\iota_v\psi(\alpha))= -v(\psi((\tilde \alpha,\tilde \beta))=0\ ,$$
hence $d(\iota_{v}\psi)(\alpha,\beta)=0$.
\begin{re} \label{closed2}
In the conditions and with the notations of Proposition \ref{diff-curv} suppose that $v\in T_y^\R Y$ is tangent to a real submanifold $Y_0\subset Y$ on which the map %
$$y\mapsto F_{{\cal L}_y,h_y}\in i Z^2(X,\R)$$
 is constant. Then $d\varphi(v)$ belongs to the image of $iH^1(X,\R)$ in $H^1(X,{\cal O}_X)$ via the canonical map $\hg$, and the element in $iH^1(X,\R)$ corresponding to $d\varphi(v)$ is the Rham class of the imaginary 1-form $\iota_{v}(F_{{\cal L},h})$ (form which is closed by Remark \ref{closed1}).
\end{re}

This is a direct consequence of Proposition \ref{diff-curv} taking into account Lemma  \ref{Lie}, formula (\ref{hg}) and  Remark \ref{closed1}.

\subsection{Meromorphic maps with values in Grassmann manifolds associated  with a holomorphic vector bundle}

\begin{pr} \label{ev} Let $Z$ be a compact complex manifold  and ${\cal F}$ a rank $r$-holomorphic vector bundle on $Z$. For  every $z\in Z$ let 
$$\mathrm{ev}_z:V:=H^0({\cal F})\to {\cal F}(z)\ .
$$
be the evaluation map at $z$. Suppose that $\ker(\ev_z)\ne\{0\}$ for all $z\in Z$, and let 
$$k:=\min\{\dim(\ker(\ev_z))|\ z\in Z\}>0$$
 be the generic (hence the minimal) dimension of these kernels. Then the map $z\mapsto  \ker(\ev_z)$ defines a   non-constant meromorphic map $\kg:Z\dasharrow  \mathrm{Gr}_k(V)$ in the Grassmann manifold of $k$-planes  of $V$.
\end{pr}
\begin{proof}  Let $Z_0\subset Z$ be the Zariski-open subset of point $z\in Z$ for which it holds $\dim(\ker(\ev_z))=k$. We obtain a holomorphic  map $\kg_0:Z_0\to  \mathrm{Gr}_k(V)$ given by $\kg_0(z):=\ker(\ev_z)$.  The subset
$$\Gamma:=\{(z,W)\in Z\times  \mathrm{Gr}_k(V)|\ W\subset\ker(\ev_z)\}
$$
of  $Z\times  \mathrm{Gr}_k(V)$ is obviously analytic, proper over $Z$, and its Zariski open subset $\Gamma\cap(Z_0\times \mathrm{Gr}_k(V))$ coincides with the graph $\mathrm{graph}(\kg_0)$ of $\kg_0$.
 The irreducible component $C$ of $\Gamma$ which contains $\Gamma\cap(Z_0\times \mathrm{Gr}_k(V))$ coincides with the closure of   $\mathrm{graph}(\kg_0)$, and is an analytic subset $ Z\times  \mathrm{Gr}_k(V)$; this proves that $\kg_0$ defines a meromorphic map $\kg:Z\dasharrow  \mathrm{Gr}_k(V)$.  The analytic map $\kg_0$ cannot be constant because  if it were, there would exist a  $k$-plane $W\subset V$ such that $W\subset \ker(\ev_z)$ for every $z\in Z$. But this contradicts of course  the obvious equality $\cap_{z\in Z} \ker(\ev_z)=\{0\}$ (a section of ${\cal F}$ which vanish at every point of $Z$  must be trivial). 
\end{proof}

\begin{co} \label{kerev} Let $Z$ be a compact complex manifold, and  ${\cal F}$ a  holomorphic vector bundle on $Z$ such  that $\ker(\ev_y)\ne\{0\}$ for all $y\in Z$. Then $a(Z)>0$. \end{co}
Recall that the condition $a(Z)>0$ means that $Z$ has non-constant meromorphic functions. Equivalently, $Z$ is covered by effective divisors (see \cite{CP} p. 331).
\begin{proof} With the notations introduced in the proof of Proposition \ref{ev}, denote by $Z'\subset  \mathrm{Gr}_k(V)$ the projection on the $\mathrm{Gr}_k(V)$ of the irreducible compact space $C$, which is the graph of the meromorphic map $\kg$. In other words $Z'$ is the image of  $\kg$. We have a surjective meromorphic map $Z\dasharrow Z'$, so, using  Proposition 2.23 in  p. 332 \cite{CP}, we obtain  $a(Z)\geq  a(Z')$. But  $a(Z')>0$ because $Z'$ is algebraic projective of non-vanishing dimension.  
\end{proof} 

\subsection{Families containing both filtrable and non-filtrable bundles}

In this section we will prove that if $X$ is a compact complex surface with $a(X)=0$, $Z$ a compact complex manifold and ${\cal E}$ a holomorphic rank 2-bundle on $Z\times X$ satisfying certain properties, then $a(Z)>0$. The main ingredient is Corollary \ref{kerev} above, which shows that the  simple existence of a holomorphic  rank 2-bundle on $Z$ whose space of sections has dimension $\geq 3$ implies $a(Z)>0$.

\begin{pr} \label{filt-non-filt}
Let $X$ be a compact complex surface with $a(X)=0$, $Z$ a compact connected complex manifold and ${\cal F}$ a holomorphic rank 2-bundle on $Z\times X$ such that  
$$\{z\in Z|\ {\cal F}_z\hbox{ is filtrable}\}\ne\emptyset\ ,\ \mathrm{int}\{z\in Z|\ {\cal F}_z\hbox{ is non-filtrable}\}\ne\emptyset\ .
$$
 If $(p_X)_*({\cal F})\ne 0$ then there exists a point $x\in X$ such that  $h^0(Z,{\cal F}^x)\geq 3$. In particular, by Corollary \ref{kerev}, one has $a(Z)>0$.

\end{pr}
\begin{proof}  For $x\in X$  denote by ${\cal F}^x$ the bundle on $Z$ defined the restriction
$\resto{{\cal F}}{Z\times\{x\}}$ and for $z\in Z$ denote as usually by ${\cal F}_z$ the bundle on $X$ defined the restriction
$\resto{{\cal F}}{\{z\}\times X}$.  

Since the sheaf $(p_X)_*({\cal F})$ is torsion free,  the hypothesis implies that $h^0({\cal F}^x)>0$ for every $x\in X$. We will prove the claim by reductio  ad absurdum. Therefore suppose $h^0(Z,{\cal F}^x)\leq 2$ for every $x\in X$. 
\\
\\ 
{\it Case 1}. One has $h^0(Z,{\cal F}^x)\leq 2$ for every $x\in X$, and $h^0(Z,{\cal F}^x)=2$ for generic $x\in X$.\\

By Grauert's semicontinuity theorem it follows that in this case $h^0(Z,{\cal F}^x)=  2$ for every $x\in X$, so  by Grauert's local triviality theorem  the sheaf ${\cal S}:=(p_X)_*({\cal F})$ is locally free of rank 2, and its fiber    ${\cal S}(x)$ at $x\in X$ is precisely $H^0({\cal F}^x)$. The evaluation map $e:p_X^*({\cal S}) \to{\cal F}$ on $Z\times X$ is a morphism of   locally free sheaves of rank 2. For any $x\in X$ the restriction 
$$e_x:=\resto{e}{Z\times\{x\}}:\resto{p_X^*({\cal S})}{Z\times\{x\}}=H^0({\cal F}^x)\otimes{\cal O}_Z\to {\cal F}^x$$
of $e$  to $Z\times\{x\}$ induces identity between the spaces of global sections so
\begin{equation}\label{ex}
e_x\ne 0\ \forall x\in X .
\end{equation}
In particular $e\ne 0$, so the rank of $e$ is either generically 2, or generically 1.\\
\\
{\it Case 1a }. The rank of $e$ is  generically 2.\\

The first step is to remove the divisorial component of the zero locus $V(e)$ of $e$. The morphism $e$ can be identified with a section $\sigma$ in the homomorphism bundle ${\cal H}:=p_X^*({\cal S})^\vee\otimes {\cal F}$. Denoting by $D_0$ the divisorial component of the vanishing locus of $\sigma_e$, it follows that $\sigma_e$ factorizes through a section $\sigma_e^0$ in the subsheaf
$${\cal H}(-D_0)=p_X^*({\cal S})^\vee\otimes ({\cal F}(-D_0))  $$ 
of ${\cal H}$. The section $\sigma_e^0$  defines a morphism 
$$e^0:p_X^*({\cal S})\to {\cal F}^0:={\cal F}(-D_0)$$
 whose vanishing locus $V(e_0)$ contains no divisorial component.

Note that the divisor $D_0$ cannot contain any horizontal  irreducible component (i.e., a component   of the form $Z\times D''$) because by (\ref{ex}) the set $V(e)$ does not contain any slice  $Z\times \{x\}$. Therefore, by Lemma \ref{hor-vert} below, one has $D_0=D'_0\times X$ for an effective divisor $D'_0\subset Z$. This shows that
\begin{equation}\label{F0}
{\cal F}^0\simeq {\cal F}\otimes p_Z^*({\cal O}_{Z}(-D'_0))\ ,
\end{equation}
which implies
\begin{equation}\label{FF_0}
{\cal F}^0_z\simeq {\cal F}_z\ \forall z\in Z\ .
\end{equation}
In other words ${\cal F}$, ${\cal F}^0$ are equivalent families of bundles on $X$ parameterized by $Z$.

The morphism 
$$\det(e_0):\wedge^2(p_X^*({\cal S}))\to \wedge^2({\cal F}_0)$$
defines a {\it nontrivial} section of a holomorphic line bundle over $Z\times X$. Let $\Delta$ be the  effective divisor defined by this section. Using Lemma \ref{hor-vert} write
$$\Delta=\Delta'\times X+Z\times \Delta''
$$
for effective divisors $\Delta'\subset Z$, $\Delta''\subset X$.
Restricting   $e^0$ to  $\{z\}\times X\simeq X$ for $z\not\in \Delta'$ we get a morphism 
$${\cal S}\to  {\cal F}^0_z\simeq {\cal F}_z$$
 whose restriction to $X\setminus \Delta''$ is an isomorphism.  Therefore, for any $z\in Z\setminus \Delta'$ the bundles ${\cal S}$, ${\cal F}_z$ have the same type (filtrable or non-filtrable). Since we assumed that ${\cal F}_z$ is non-filtrable for any $z$ in a non-empty open subset of $Z$, it follows that ${\cal S}$ is necessarily non-filtrable.

If we take now a point $z\in \Delta'$ we obtain   a morphism
$${\cal S}\to  {\cal F}^0_z\simeq {\cal F}_z
$$
whose rank this time is $\leq 1$ (because $\det(e^0)$ vanishes on $\Delta'\times X$), but is certainly non-zero for generic $z\in \Delta'$ because the vanishing locus of $e^0$ contains no divisorial component. Therefore, if $\Delta'$ were non-empty, it would follow  that ${\cal S}$ is filtrable, because it contains the rank 1 kernel of the morphism ${\cal S}\to  {\cal F}^0_z$ for generic $z\in \Delta'$. This would lead to a contradiction, hence  $\Delta'=\emptyset$.  But this implies that ${\cal F}_z$ is non-filtrable (as is ${\cal S}$) for every $z\in Z$, which is impossible, because some of these bundles are  filtrable by hypothesis.
\\ 
\\
{\it Case 1b }. The rank of $e$ is  generically 1.\\

In this case the morphism ${\cal S}\to   {\cal F}_z$ obtained by restricting $e$ to $\{z\}\times X$ is non-zero (and with rank 1 kernel) for generic $z\in Z$. But this would imply that    the bundles ${\cal F}_z$ are filtrable for generic $z\in Z$, which contradicts our hypothesis. 
\\ \\
{\it Case 2}. One has $h^0(Z,{\cal F}^x)\leq 1$ for generic $x\in X$.\\

In this case ${\cal S}:=R^0(p_X)_*({\cal F})$ is a  torsion free rank 1 sheaf on $X$. Let $X_0\subset X$ the dense Zariski open subset of points $x\in X$ for which $h^0({\cal F}^x)=1$. The sheaf ${\cal S}$ is locally free of rank 1 on $X_0$, and ${\cal S}(x)=H^0({\cal F}^x)$ for every $x\in X_0$ by Grauert's local triviality and base change theorems.  

The restriction $e_x$ of the evaluation map $e:p_X^*({\cal S})\to {\cal F}$ to $Z\times\{x\}$  is non-zero for every $x\in X_0$, hence  the restriction $\resto{e}{Z\times X_0}:\resto{{\cal S}}{Z\times X_0}\to \resto{{\cal F}}{Z\times X_0}$ does not vanish identically, hence its zero set   is a  proper analytic subset of $Z\times X_0$. This shows that, for generic $z\in Z$, the induced map $\resto{\cal S}{X_0}\to \resto{{\cal F}_z}{X_0}$ is non-zero. Therefore,   for generic $z\in Z$, ${\cal F}_z$ contains the rank 1-subsheaf $e_z({\cal S})$, so it is filtrable, which again contradicts the assumption.
\end{proof}

\begin{lm}\label{hor-vert}
Let $Z$, $X$ be connected, compact, complex manifolds, and suppose that $X$ has only finitely many irreducible divisors. Then any irreducible effective divisor $D\subset Z\times X$ has either the form $D'\times X$ for an irreducible effective divisor $D'$ of $Z$, or the form $Z\times D''$ for an irreducible effective divisor $D''$ of $X$.
\end{lm}
In other words, any irreducible effective divisor $D\subset Z\times X$ is either vertical or horizontal.

\begin{proof}
 If $p_Z(D)=Z$, then the fiber of the map $\resto{p_Z}{D}:D\to Z$ is an effective divisor of $X$ for every $z$ in a non-empty Zariski open subset $Z_0\subset Z$ (see Theorem \cite{Fi} p. 140). Therefore $D\cap(Z_0\times X)$ is flat over $Z_0$ (see Lemma 1 \cite{DT}), so it defines a holomorphic map $Z_0\to {\cal D}ou(X)$, which must be constant, because our assumption on $X$ implies that the Douady moduli space ${\cal D}ou(X)$ of effective divisors of $X$ is 0-dimensional at every point. Let $D''\subset X$ be the effective divisor corresponding to this constant. We have $Z_0\times D''= D\cap(Z_0\times X)\subset D$, hence (since $D$ is closed) one has $Z\times D''\subset D$. Since $D$ is irreducible this implies that $D''$ is irreducible and $Z\times D''= D$.
 
 If $p_Z(D)\ne Z$ the underlying analytic set of $p_Z(D)$ will be an irreducible analytic subset $H\subset Z$ of pure codimension 1. The underlying analytic set of $D$ is contained in $H\times X$, hence (since $H\times D$ is irreducible) it coincides with it.
 \end{proof}


\begin{thebibliography}{BXDB}

 
\bibitem[Ba]{Ba} Barlet, D.: {\it How to Use the Cycle Space in Complex Geometry}, Several Complex Variables MSRI Publications Volume 37 (1999). 
 
  

 \bibitem[BHPV]{BHPV}  Barth, W;    Hulek, K.;  Peters, Ch.;   Van de Ven, A.:
{\it Compact complex surfaces}, Springer   (2004).

\bibitem[BS]{BS}    Bănică, C.; Stănășilă, O.: {\it Algebraic methods in the global theory of complex spaces} Wiley, New York (1976).

 \bibitem[Bi]{Bi} Bismut, Jean-Michel: {\it Hypoelliptic Laplacian and Bott-Chern cohomology}, Progress in Mathematics, Volume 305, Springer (2013).
 
 \bibitem[Bu]{Bu} Buchdahl, N.: {\it A Nakai-Moishezon criterion for non-Kähler surfaces}, Ann. Inst. Fourier Vol. 50,  1533–1538 (2000).
 
 \bibitem[CP]{CP} Campana, F.; Peternell, Th.: {\it Cycle Spaces}  in Encyclopaedia of Mathematical Sciences Vol. 74, Several
Complex Variables VII (Grauert, H;  Peternell, Th.; Remmert R. Eds) 319-349, Springer (1994).
 
 \bibitem[D]{D}   Donaldson, S. :   {\it Floer homology groups in Yang–Mills theory}, Cambridge University Press (2004).
 
 
\bibitem[DK]{DK}   Donaldson, S.;   Kronheimer, P.:   {\it The Geometry of
Four-Manifolds}, Oxford Univ. Press (1990).

\bibitem[D1]{D1} G. Dloussky: {\it Structure des surfaces de Kato}, Mém. Soc. Math. Fr., Tome 112 n$^0$ 14 (1984).

  
\bibitem[D2]{D2}   G. Dloussky: {\it From non-Kählerian surfaces to Cremona group of $P^2(\C)$}, {\tt arXiv:1206.2518} [math.CV].

\bibitem[DOT]{DOT} Dloussky, G.; Oeljeklaus, K.; Toma, M.: {\it Class $\rm VII\sb 0$
surfaces with $b\sb 2$ curves},  Tohoku Math. J. (2) 55
 no. 2, 283-309  (2003).
 
  \bibitem[DT]{DT} Dloussky, G.; Teleman, A.: {\it Infinite bubbling in non-Kählerian geometry}, Math. Ann. Vol. 353,   No. 4,  1283-1314,  (2012)
 
 \bibitem[Dou]{Dou} Douady, A.: {\it Le problème des modules locaux pour les espaces $\C$-analytiques compacts}, Ann. Sci. \'Ecole Norm. Sup. Série 4 t. 7, 569-602 (1974).
 

 \bibitem[DV]{DV} Douady, A., Verdier, J.-L.: {\it  Séminaire de Géométrie analytique de l’Ecole Normale Supérieure}, 1971-72. Astérisque 16, Soc. Math. France (1974).
 
\bibitem[E]{E}   Enoki, I.: {\it Surfaces of class VII$_0$ with curves}, Tohoku Math. J. 33, 453–492 (1981).  

\bibitem[EL]{EL} Eisenbud, D.; Levine, H.: {\it An algebraic formula for the degree of a ${\cal C}^\infty$ map germ},  Annals of Mathematics, Vol. 106 19-44 (1977).
 
  \bibitem[Fi]{Fi} Fischer, G.: {\it Complex Analytic Geometry}, Springer Verlag 538 (1976).

\bibitem[Fl]{Fl} Flenner, H. {\it Ein Kriterium für die Offenheit der Versalität}, Math. Z. 178, 449-473 (1981).

\bibitem[Go]{Go} Godement, R.: {\it Topologie algébrique et théorie des faisceaux}, Hermann (1973).

\bibitem[GR]{GR} Grauert, H. Remmert, R. : {\it Coherent Analytic Sheaves}, Spinger-Verlag (1984).

\bibitem[Gri]{Gri}   Grivaux, J.:   {\it  Variation of the holomorphic determinant bundle}, {\tt arXiv:1205.6170} (2012).

\bibitem[Ho]{Ho} N. Hoffmann: {\it The Moduli Stack of Vector Bundles on a Curve}, Teichmüller Theory and Moduli Problems,  I. Biswas, R. S. Kulkarni, S. Mitra (Eds.), Ramanujan Mathematical Society Lecture Notes Series,  pp. 387-394.Vol. 10 (2010).

 
\bibitem[KK]{KK} L.  Kaup, B. Kaup: {\it  Holomorphic Functions of Several Variables. An Introduction to the Fundamental Theory}, de Gruyter (1983).
 
\bibitem[K]{K}   Kobayashi, S.: {\it Differential geometry
of complex vector bundles.},
Princeton Univ. Press (1987).

\bibitem[KS]{KS} Kodaira, K.; Spencer, D. C.: {\it A theorem of completeness for complex analytic fiber spaces}, Acta Mathematica Vol. 100, Issue 3-4, pp 281-294 (1958).


\bibitem[Ko]{Ko} Kobayashi, N.: {\it Differential Geometry of Complex Vector Bundles}, Princeton University Press (1987).

\bibitem[LT]{LT} M. Lübke, A. Teleman: {\it The Kobayashi-Hitchin correspondence}, World Scientific Publishing  (1995).

  
\bibitem[Na1]{Na1} Nakamura, I.: {\it On surfaces of class $VII_0$ surfaces with
curves}, Invent. Math. 78, 393-443 (1984).

\bibitem[Na2]{Na2} Nakamura, I.: {\it Towards classification of non-K\"ahlerian
surfaces}, Sugaku Expositions vol. 2, No 2 , 209-229 (1989).


\bibitem[Na3]{Na3} Nakamura, I.: {\it On surfaces of class $VII_0$ surfaces with
curves II}, T\^ohuku Mathematical Journal vol 42, No 4, 475-516 (1990).


\bibitem[OT]{OT} Oeljeklaus, K.; Toma, M.: {\it Logarithmic Moduli Spaces for Surfaces of Class VII},  Math. Ann 341 (2), pp. 323-345 (2008).

\bibitem[Re]{Re} Remmert, R.: {\it Local Theory of Complex Spaces} in Several
Complex Variables VII, Sheaf-Theoretical Methods in Complex Analysis, H. Grauert, Th. Peternell, R. Remmert (Eds.), Encyclopaedia of Mathematical Sciences Vol. 74, Springer, (1994).
%
\bibitem[Sp]{Sp}  Spanier, E.: {\it Algebraic Topology}, McGraw–Hill, Inc. (1966).
  
\bibitem[Te1]{Te1} Teleman, A.:  {\it Projectively flat surfaces and Bogomolov's
theorem on class $VII_0$ - surfaces},  Int. J.	 
Math., Vol.5, No 2, 253-264 (1994).

\bibitem[Te2]{Te2} Teleman, A.: {\it Donaldson theory on non-K\"ahlerian
surfaces and class VII surfaces with $b_2=1$}, Invent. math. 162, 493-521 (2005).

\bibitem[Te3]{Te3} Teleman, A.: {\it The pseudo-effective cone of a non-Kählerian surface and applications},  Math. Ann. Vol. 335, No 4, 965-989, (2006).

\bibitem[Te4]{Te4} Teleman, A.: {\it Instantons and holomorphic curves on class VII surfaces}, Annals of Mathematics 172,  1749-1804 (2010).

\bibitem[Te5]{Te5} Teleman, A.: {\it Gauge theoretical methods in the classification of non-Kählerian surfaces},
"Algebraic Topology - Old and New" (the Postnikov Memorial Volume), Banach Center Publications Vol. 85 (2009).

\bibitem[Te6]{Te6} Teleman, A.: {\it On the torsion of the first direct image of a locally free sheaf}, preprint, august 2013.
 
\end{thebibliography}
\end{document}